\documentclass[3p,12pt]{elsarticle}
\usepackage{amsmath}
\usepackage{amsfonts}
\usepackage{geometry}                
\usepackage[parfill]{parskip}    
\usepackage{graphicx}
\usepackage{amssymb}
\usepackage{epstopdf}
\usepackage{psfrag}
\newcommand{\figref}[1]{{Figure~\ref{#1}}}
\newcommand {\be}{\begin{equation}}
\newcommand {\bes}{\begin{displaymath}}
\newcommand {\es}{\end{displaymath}}
\newcommand {\e}{\end{equation}}

\newtheorem{example}{\underline{Example}}

\newcommand {\bea}{\begin{eqnarray}}
\newcommand {\ea}{\end{eqnarray}}
\newtheorem{proposition}{Proposition}[section]
\newtheorem{theorem}{Theorem}[section]
\newtheorem{Assumption}{Assumption}[section]
\newtheorem{lemma}{Lemma}[section]
\newtheorem{remark}{Remark}[section]
\newtheorem{definition}{Definition}[section]

\newcommand{\defref}[1]{{Definition~\ref{#1}}}
\newenvironment{proof}[1][Proof]{\textbf{#1.} }{\hspace{\stretch{1}}\rule{0.5em}{0.5em}}

\usepackage{xspace}
\newcommand{\Leja}{{L\'{e}ja}\xspace}

\newcommand{\Dt}{\Delta t}
\usepackage{subfigure}

\newcommand{\thmref}[1]{{Theorem~\ref{#1}}}
\newcommand{\lemref}[1]{{Lemma~\ref{#1}}}
\newcommand{\secref}[1]{{Section~\ref{#1}}}
\newcommand{\assref}[1]{{Assumption~\ref{#1}}}
\newcommand{\propref}[1]{{Proposition~\ref{#1}}}

\journal{Computers \& Mathematics with Applications}

\begin{document}
\begin{frontmatter}
\title{A\MakeLowercase{N  EXPONENTIAL INTEGRATOR FOR FINITE VOLUME DISCRETIZATION OF NONLINEAR PARABOLIC PARTIAL DIFFERENTIAL EQUATION}}
\author[at,atb]{Antoine Tambue}
\ead{antonio@aims.ac.za}
\address[at]{The African Institute for Mathematical Sciences(AIMS) and Stellenbosh University,
6-8 Melrose Road, Muizenberg 7945, South Africa}
\address[atb]{Center for Research in Computational and Applied Mechanics (CERECAM), and Department of Mathematics and Applied Mathematics, University of Cape Town, 7701 Rondebosch, South Africa.}

\begin{abstract}
We consider the numerical approximation of a general second order
 semi--linear parabolic  partial differential equation. Equations of
this type arise in many contexts, such as  transport in porous media which  is fundamental in many geo-engineering applications, including oil and gas recovery from subsurface.
Using the finite volume  with  two-point flux approximation on regular mesh combined with exponential 
time differencing of order one (ETD1) for temporal discretization, we derive the $L^{2}$ estimate
under the assumption that the non linear term is locally  Lipschitz.  Numerical simulations to sustain the theoretical results are provided.

\end{abstract}
\end{frontmatter}

\section{Introduction}

Flow and transport are fundamental in many geo-engineering applications, including oil and gas recovery from
hydrocarbon reservoirs, groundwater contamination and sustainable use
of groundwater resources, storing greenhouse gases (e.g. CO$_2$) or
radioactive waste in the subsurface, or mining heat from geothermal
reservoirs.
 In porous media, a non-degenerated advection-diffusion-reaction is given by 
\begin{eqnarray}
\label{adr}
  \dfrac{\partial X}{\partial t}=\nabla \cdot(\mathbf{D}(\mathbf{x}) \nabla  X)-\nabla \cdot(\mathbf{q} (\mathbf{x}) X)+f(\mathbf{x},X)
 \quad (\mathbf{x},t)\,\in \Omega \times \left[ 0,T\right],
\end{eqnarray}
where  $\Omega$ is  an open domain  of $ \mathbb{R}^{d},\;d\in \{2,3 \}$,  $\mathbf{D}$ 
is the symmetric dispersion tensor, $X$ is the unknown concentration of the contaminant, $\mathbf{q}$ the Darcy's velocity and $f$ 
 the reaction and source term.  For the sake of simplicity, without loss of generality, we assume that $f$ is explicitly independent of time. The  model equation \eqref{adr}  finds interest in many engineering problems
with specific coefficients. 
Finite element, finite volume or the combination finite element-finite volume methods are mostly used for space discretization 
of the problem \eqref{adr}  while explicit, semi implicit and fully
implicit methods are usually used for time discretization (see \cite{Stig,FV,AM,Afif,EFV}).
Due to time step size constraints, fully implicit schemes are more popular 
for time discretization  for quite a long time compared to explicit Euler schemes. However, implicit schemes need at each 
time step a solution of large systems of nonlinear equations. This can be the bottleneck in computations. In recent years, 
exponential integrators have become an attractive
 alternative in many evolutions equations (see \cite{Antoine,TLGspe,TLGspe1,ATthesis, ost,LEJAK,LE,LE1,LEJAK,LEJAK1,Tr}).
In contrast to classical methods, they are robust with respect to the P\'{e}clet number, they do not require the solution of large linear systems.
Instead they make explicit use of the
matrix exponential and related matrix functions. The price to pay 
is  the computing of the matrix
exponential functions of the non diagonal matrices,
 which  has revived interest and significance progresses nowadays 
(see \cite{ost,ATthesis,LEJAK,LE1,LEJAK,LEJAK1}).

In this work, we combine a finite volume method with the first order 
exponential time differencing scheme of order 1 (ETD1).
Although both discretization techniques 
have been together used 
for solving evolutionary  problems like \eqref{adr} (see \cite {Antoine,TLGspe,TLGspe1}), a proper 
combination of rigorous  convergence proof of them has been lacking so far.  Furthermore  the nonlinear term is assumed to be locally  Lipschitz, which covers many reaction functions in geo-engineering applications.

The paper is organised as follows. In \secref{sec:semi}, we
present the semi group formulation of \eqref{adr}, the existence and uniqueness of  
the solution along with some proprieties of  the mild solution.
 In \secref{sec:FVm}, we present the  finite volume space discretization of \eqref{adr}, 
 the existence and  the uniqueness of  the corresponding semi-discrete problem, and  
 the $L^{2}$  error estimate between the exact solution and the semi-discrete solution.
We end by presenting in \secref{sec:etd} The time discretization  of the semi-discrete problem based  on ETD1 
scheme is presented in \secref{sec:etd}, along with the convergence proof of  the fully discrete scheme based 
on  finite volume method and ETD1 scheme.  We end by providing numerical simulations to sustain the theoretical results in \secref{sim}. 
These results also show the efficiency of the ETD1 scheme compared to the standard time integrators, from which ETD1  scheme is ten times faster that the standard implicit scheme.

\section{Semi group formulation and well posedness}
\label{sec:semi}
Let us start by presenting briefly the notation of the main function
spaces and norms that we will use in this paper. 
We denote by $\Vert \cdot \Vert$ the norm associated to
the inner product $(\cdot ,\cdot )$ of the Hilbert space $H=L^{2}(\Omega)$. The norms in the Sobolev spaces $H^m(\Omega),\, m \geqslant 0$ will be denoted by
$\Vert. \Vert_m$. The space $H^{-1}(\Omega)$ is  the dual of $H_0^{1}(\Omega)$ equipped with the norm $ \Vert u\Vert_{-1}=\underset{ v\in H_0^{1}(\Omega)}{\sup}\dfrac{\vert (u,v) \vert}{\Vert u\Vert_{1}}$.
For a Banach space $\mathcal{V}$ we denote by 
$L(\mathcal{V})$  the set of bounded linear mapping  from
$\mathcal{V}$ to $\mathcal{V}$. 
 We assume that the domain $\Omega$ is bounded, has a smooth  boundary or is a convex polygon.
For the sake of simplicity, without loss of generality, we use the homogeneous Dirichlet boundary condition. 
We  also assume that  the Darcy velocity $\mathbf{q}$ is known, and satisfies the
 mass conservation for  incompressible fluids without internal source, that is $\nabla\cdot \mathbf{q}=0.$

 For  a given initial solution $X_0 \in H$, the  model problem \eqref{adr} is reformulated as:
find the function
 $X(t)\in H^{1}(\Omega)$ such that
\begin{eqnarray}
\label{adr6}
\begin{cases}
 \partial X/ \partial t + \mathcal{A} X= f(\mathbf{x},X) \quad \quad  \quad \quad&  (\mathbf{x},t) \in \Omega \times \left[0,T \right]  \\
 X(\mathbf{x},0)=X_{0} \quad \quad \quad \quad & \mathbf{x}  \in \Omega\\
X(\mathbf{x},t) =0   \quad \quad \quad \quad & (\mathbf{x},t) \in \partial \Omega \times \left[0,T \right], 
\end{cases}
 \end{eqnarray}
where
\begin{eqnarray*}
  \mathcal{A}X = \mathcal{A}( \mathbf{x})X &=&- \nabla\cdot\left( \mathbf{D}\nabla X\right) +\nabla \cdot (\mathbf{q}(\mathbf{x})X)\\
  &=-& \underset{i,j=1}{\sum^{d}}\dfrac{\partial }{\partial x_{i}}\left( D_{i,j}(\mathbf{x})\dfrac{\partial X
 }{\partial x_{j}}\right) + \underset{i=1}{\sum^{d}}q_{i}(\mathbf{x})\dfrac{\partial X
 }{\partial x_{i}}+ \left(\nabla \cdot \mathbf{q}\right) X \\
&=-&\underset{i,j=1}{\sum^{d}}\dfrac{\partial }{\partial x_{i}}\left( D_{i,j}(\mathbf{x})\dfrac{\partial X
 }{\partial x_{j}}\right) + \underset{i=1}{\sum^{d}}q_{i}(\mathbf{x})\dfrac{\partial X
 }{\partial x_{i}}.
\end{eqnarray*}
For well posedness of \eqref{adr6}, 
 we assume that $\mathbf{D}$ is symmetric, $D_{i,j} \in L^{\infty}(\Omega), q_{i}\in L^{\infty}(\Omega)$ and  there exists a positive constant $c_{1}>0$ such that 
\begin{eqnarray}
\label{ellipticity}
\underset{i,j=1}{\sum^{d}}D_{i,j}(\mathbf{x})\xi_{i}\xi_{j}\geq c_{1}\vert \xi \vert^{2}  \;\;\;\;\;\;\forall \xi \in \mathbb{R}^{d}\;\;\; \mathbf{x} \in \overline{\Omega}\;\;\; c_{1}>0,
\end{eqnarray}
and 
\begin{eqnarray}
\label{lipf}
 \vert f(\mathbf{x},u)-f(\mathbf{x},v) \vert \leq L \left( 1 + \vert u\vert^{\gamma}+ \vert v\vert^{\gamma}\right) \vert u- v \vert\;\;\;\forall u, v \in \mathbb{R}\;\; x \in \overline{\Omega},\;t\in [0,T], 
\end{eqnarray}
or 
\begin{eqnarray}
\label{lipff}
 \vert f(\mathbf{x},u)-f(\mathbf{x},v) \vert \leq L \left( 1+ u+v+ \vert u\vert^{\gamma}+ \vert v\vert^{\gamma}\right) \vert u- v \vert\;\;\;\forall u, v \in \mathbb{R}\;\; x \in \overline{\Omega},\;t\in [0,T], 
\end{eqnarray}
with  $\gamma =2 $ for $d=3$ and  $\gamma \in \left[ 0, \infty \right) $  for $d=2$.

Set  $V=H^{1}_{0}(\Omega)$, the bilinear form associated to the operator $\mathcal{A} $ is given by
\begin{eqnarray}
\label{tvar}
a(u,v)=\int_{\Omega}\left(\underset{i,j=1}{\sum^{d}} D_{i,j}\dfrac{\partial u}{\partial x_{j}} \dfrac{\partial v}{\partial x_{i}}+\underset{i=1}{\sum^{d}}q_{i} \dfrac{\partial u}{\partial x_{j}}v\right)dx\;\;\;\;\;\;\; u, v \in V. 
\end{eqnarray}
According to G\aa{}rding' s inequality (see \cite{ATthesis,lions}), there exists two positive constants $c_{0}$ and $\lambda_{0}$  such that
\begin{eqnarray}
 \label{coer}
  a(v,v)+c_{0}\Vert v\Vert^{2}\geq  \lambda_{0}\Vert v\Vert_{1}^{2}\;\;\; \quad \quad \forall v\in V.
\end{eqnarray}
By adding $c_{0}X$ in both side of the first equation of \eqref{adr} we have a new operator that we still call $\mathcal{A}$  corresponding to the new bilinear form that we still call 
$a$ such that the following  coercivity property holds
\begin{eqnarray}
\label{ellip}
a(v,v)\geq \;\lambda_{0} \Vert v\Vert_{1}^{2}\;\;\;\;\;\forall v \in V.
\end{eqnarray}
 For sake of simplicity, we will still call the right hand side of the first equation of \eqref{adr6} $f$. We define the following  Nemytskii operator $F: H\rightarrow H$  by 
 \begin{eqnarray}
 \label{N}
(F(X))(x)= f(x,X). 
 \end{eqnarray}


Using the Green's formula, the  weak  form of \eqref{adr6} consists to find the function  $X(t) \in V$ such that
\begin{eqnarray}
\label{var6}
\begin{cases}
 (X_{t},\chi) + a(X,\chi) = (F(X),\chi) \quad \quad  \quad \quad&   \forall \chi \in V ,\quad  t \in \left[0,T \right]  \\
X(0)=X_{0}.
\end{cases}
\end{eqnarray}
Note that $a(,)$ is bounded in $V\times V$, so the following  operator $A: V\rightarrow V^{*}$ is well  defined by Riez's representation Theorem

\begin{eqnarray}
 a(u,v)=\langle Au, v\rangle,\,\,\,\, \forall u, v \in V,
\end{eqnarray}
where $V^{*}$ is the adjoint space (or dual space) of $V$ and $\langle ,\rangle$ the duality pairing between $V^{*}$ and $V$. By  identifying $H$ to its adjoint space $H^{*}$, 
we get the following  continuous and dense inclusions
\begin{eqnarray}
\label{gerland}
 V\subset H \subset V^{*}.
\end{eqnarray}
So, we have 
\begin{eqnarray}
\label{gerland1}
 (u,v)=\langle u, v \rangle \qquad \qquad \qquad \forall u\in  H,\, \forall  v\in V.
\end{eqnarray}
The domain of $A$ denoted by  $\mathcal{D}(A)$ is defined by 
\begin{eqnarray}
 \mathcal{D}(A)= \{ u \in V,\, Au \in H \}. 
\end{eqnarray}
We write the restriction of $A: V\rightarrow V^{*}$ to $\mathcal{D}(A)$ again by  $A$, which is therefore regarded as an operator of $H$ (more precisely the $H$ realization of $\mathcal{A}$ \cite[p. 812]{lions}).
In the abstract setting,  equation \eqref{var6}  is equivalent  to find  the function  $X(t) \in V$ such that

\begin{eqnarray}
\label{var}
\begin{cases}
 X_{t} + AX = F(X) ,\quad \qquad \qquad  t \in \left[0,T \right]  \\
X(0)=X_{0},
\end{cases}
\end{eqnarray}
where equation \eqref{var} is understood in  the space $V^{*}$ using \eqref{gerland} and \eqref{gerland1}.
As the domain $\Omega$ has a smooth  boundary or is a convex polygon,  we therefore have (see \cite{lions})
\begin{eqnarray}
\label{dom}
  \mathcal{D}(A)= H_{0}^{1}(\Omega)\cap H^{2}(\Omega).
 \end{eqnarray}
The $ V-$ellipticity (\ref{ellip}) implies that $-A$ is a sectorial on $H=L^{2}(\Omega)$ (see \cite{Henry,lions}) i.e.  there exists $C_{1},\,\theta \in (\frac{1}{2}\pi,\pi)$ such that
\begin{eqnarray}
 \Vert (\lambda I +A )^{-1} \Vert_{L(H)} \leq \dfrac{C_{1}}{\vert \lambda \vert }\;\quad \quad 
\lambda \in S_{\theta},
\end{eqnarray}
where $S_{\theta}=\left\lbrace  \lambda \in \mathbb{C} :  \lambda=\rho e^{i \phi},\; \rho>0,\;0\leq \vert \phi\vert \leq \theta \right\rbrace $.

 Then  $-A$ is the infinitesimal generator of bounded analytic semigroups $S(t):=e^{-t A}$  on $L^{2}(\Omega)$  such that
\begin{eqnarray}
S(t):= e^{-t A}=\dfrac{1}{2 \pi i}\int_{\mathcal{C}} e^{t\,\lambda}(\lambda I + A)^{-1}d \lambda,\;\;\;\;\;\;\;
 \;t>0
\end{eqnarray}
where $\mathcal{C}$  denotes a path that surrounds the spectrum of $-A $.

The coercivity property in \eqref{ellip} implies  also that the set of the real part of the spectrum of  $A$ is non negative, which allows 
the definition of the fractional power of $A$
as:  for any  $\alpha>0$
\begin{eqnarray}
\left\lbrace \begin{array}{l}
 A^{-\alpha} = \dfrac{1}{\Gamma (\alpha)} \int_{0}^{\infty}t^{\alpha-1} e^{-A t} dt\\
\newline\\
A^{\alpha}=  \left(A^{-\alpha}\right)^{-1}
\end{array}\right.
\end{eqnarray}
 where $\Gamma (\alpha)$ is the Gamma function of $\alpha $ (see \cite{Henry}). It is  well known that  $\Vert.\Vert_{\alpha}\equiv \Vert A^{\alpha/2}. \Vert$ 
in the space $\mathcal{D}(A^{\alpha/2})$, and that $V=\mathcal{D}(A^{1/2})=\mathcal{D}(A^{*\,1/2})$ (see \cite{Stig,Lionsj}). Note that $A^{*\,1/2}$ is the adjoint of $A^{\,1/2}$.

For the nonlinear reaction term we make the following assumption
\begin{proposition} \textbf{[Lipschitz condition  for $F$]}\\
\label{alip}
Under  the  assumption  \eqref{lipf} or \eqref{lipff} on the  nonlinear  function $f$, let $F$ the  Nemytskii operator  corresponding to $f$ defined by \eqref{N}. 
 For each bounded set $\mathcal{B}\subset V$ there is a constant $C(\mathcal{B})$ such that  
 \begin{eqnarray}
 \label{lip}
\Vert F(u)-F(v) \Vert_{-1} \leq C(\mathcal{B})\Vert u-v\Vert,\;\;\;\;\;\forall u,v \in \mathcal{B}\\\
\label{lip1}
  \Vert F(u)-F(v) \Vert \leq C(\mathcal{B})\Vert u-v\Vert_{1},\;\;\;\;\;\forall u,v \in \mathcal{B}\ .
\end{eqnarray}
\end{proposition}
\begin{proof}
 The proofs of \eqref{lip} and \eqref{lip1} can  be found in \cite{Stig} for function $f$ of type  \eqref{lipf}.
 The proofs for  function $f$ of  type \eqref{lipff} can easily be deducted.
 Indeed using Holder inequality yields
 { \small {
 \begin{eqnarray}
  \lefteqn {\Vert F(u)-F(v)\Vert}\\
  &\leq& C \left(\left(\Vert u \Vert_{L^{q_{1}}(\Omega)} + \Vert u \Vert_{L^{q_{1}}(\Omega)}\right) \Vert u-v\Vert_{L^{p_1}(\Omega)} +\left(1+\Vert u \Vert_{L^q(\Omega)}^{\gamma}+ \Vert v \Vert_{L^q(\Omega)}^\gamma\right) \Vert u-v\Vert_{L^p(\Omega)}\right) \nonumber
 \\
  && \qquad \qquad \qquad \qquad \qquad \qquad \qquad \qquad \qquad \qquad \qquad \qquad \qquad \qquad \qquad \qquad \qquad \forall u, v \in \mathcal{B},\nonumber
 \end{eqnarray}
 }}
 where $\frac{1}{p}+\frac{\gamma}{q}=\frac{1}{2}$ with $p=q=6$ if $d=3$,   and arbitrary $p \in (1, \infty) $   if $d = 2$, 
  $\frac{1}{p_1}+\frac{1}{q_1}=\frac{1}{2}$ with $p_1 \in [3,6] $\footnote{The corresponding $q_1$ is in  the same interval} if $d=3$ and arbitrary $p_1 \in (2, \infty) $  if $d = 2$ .
  Since $\Omega$ is bounded,  $L^r(\Omega)\hookrightarrow L^s(\Omega)$  for $r\geq s$, combining with Sobolev  the embedding theorem, we  therefore  have
  \begin{eqnarray}
   \Vert F(u)-F(v)\Vert  &\leq& C \left(1+ \Vert u \Vert_{1}+\Vert v \Vert_{1}+\Vert u \Vert_{1}^{\gamma}+ \Vert v \Vert_{1}^\gamma\right) \Vert u-v\Vert_{1}.                            
  \end{eqnarray}
 Lipschitz  condition \eqref{lip} is proved  in the same manner using  the one in \cite{Stig} for  function $f$ of  type \eqref{lipf}.
\end{proof}

By Duhamel's principle we may represent  the solution of \eqref{var}  by the following  integral equation
\begin{eqnarray}
\label{mild}
 X(t)=S(t)X_{0}+ \int_{0}^{t}S(t-s)F(X(s))ds,\;\;\;\;\  t \in \left[0,T \right].
\end{eqnarray}
  \begin{theorem}
  \label{th1e}
   Assume that  $-A$ is the  infinitesimal generator of bounded analytic semigroup ($\mathbf{D}$ is symmetric, $D_{i,j} \in L^{\infty}(\Omega), q_{i}\in L^{\infty}(\Omega)$ 
    and the inequality \eqref{ellipticity} is fulfilled) and \eqref{lipf} (or \eqref{lipff}) is  satisfied.
  For any bounded set $ \mathcal{B}_{0} \subset V$ there  is $t^{*} = t^{*} ( \mathcal{B}_{0})$ 
  such that equation (\ref{mild})   has an unique solution  $ X \in C([0,t^{*}],H^{1}(\Omega)$ for any $X_{0} \in \mathcal{B}_{0}$.                      
  \end{theorem}
\begin{proof}
 Applying the contraction mapping principle in the topology of
 the Banach space $C([0,T],H^{1}(\Omega))$
to the integral equation (\ref{mild}) \cite[Theorem 3.3.3]{Henry} or \cite[Theorem 6.3.1]{pazy}
 ensure the existence and uniqueness of $X$.
 \end{proof}
 \begin{remark}
 The regularity of the solution $X$ depends of  the regularity of  the coefficients $ \mathbf{D}= (D_{i,j})_{1\leq i,j\leq d},\mathbf{q}=(q_i)_{1\leq i\leq d}$ as we can observe in \cite{evans}.
 \end{remark}

 The following proposition will be largely used in this work
 \begin{proposition}
\textbf{[Smoothing properties of the semi group \cite{Henry}]}\\
\label{prop1}
 Let $\beta \geq 0 $ and $0\leq \gamma \leq 1$, then  there exists $C>0$ such that
\begin{eqnarray*}
 \Vert A^{\beta}S(t)\Vert_{L(H)} &\leq& C t^{-\beta}\;\;\;\;\; \text {for }\;\;\; t>0,\\
  \Vert A^{-\gamma}( \text{I}-S(t))\Vert_{L(H)} &\leq& C t^{\gamma} \;\;\;\;\; \text {for }\;\;\; t\geq0.
\end{eqnarray*}
In addition, the following results hold
$$A^{\beta}S(t)= S(t)A^{\beta}\quad \text{on}\quad \mathcal{D}(A^{\beta} ).$$
$$\text{If}\;\;\; \beta \geq \alpha \quad \text{then}\quad
\mathcal{D}(A^{\beta} )\subset \mathcal{D}(A^{\alpha} ).$$
$$\Vert D_{t}^{l}S(t)v\Vert_{\beta}\leq C t^{-l-(\beta-\alpha)/2} \,\Vert v\Vert_{\alpha},\;\; t>0,\;v\in  \mathcal{D}(A^{\alpha/2})\;\; l=0,1,$$
where $ D_{t}^{l}:=\dfrac{d^{l}}{d t^{l}}$, $\Vert .\Vert_{\alpha}:= \Vert A^{\alpha/2}.\Vert$.
\end{proposition}
The following lemma will be also used in our errors estimates.
 \begin{lemma}
\label{lemmad1}
Let $X$ be the  mild solution  of (\ref{adr6}) given  in (\ref{mild}).
Let $\mathcal{B}\subset V$ be a bounded set such that $ \forall t\in[0,t^{*}(\mathcal{B})], \, X(t) \in \mathcal{B}$.
Let  $t_{1}, t_{2} \in [0,T]\subset [0,t^{*}(\mathcal{B})],\;t_{1}<t_{2}$, the following estimates hold :
\begin{itemize}
 \item (i)
If $X_{0} \in \mathcal{D}(A)$ then
\begin{eqnarray*}
\Vert X(t_{2})- X(t_{1}) \Vert &\leq&  C (\mathcal{B})
(t_{2}-t_{1})^{1-\epsilon} \left(\Vert
  X_{0}\Vert_{2}+1
     \right),
\end{eqnarray*}
for $\epsilon \in (0,1/4)$ small enough.
\item (ii)  If $X_{0} \in \mathcal{D}(A)$ and $F$ satisfies the Lipschitz condition in \eqref{lip} then
\begin{eqnarray*}
\Vert X(t_{2})- X(t_{1}) \Vert &\leq&  C(\mathcal{B})
(t_{2}-t_{1}) \left(\Vert
  X_{0}\Vert_{2}+1)
     \right).
\end{eqnarray*}
\end{itemize}
\end{lemma}
\begin{proof}
\textbf{Part (i)}.

Consider the difference
\begin{eqnarray}
\label{ddif}
\lefteqn{  X(t_{2})- X(t_{1})} & &\nonumber\\
&=&  \left(S(t_{2})-S(t_{1})\right)X_{0}+\left(
  \int_{0}^{t_{2}}S(t_{2}-s)F(X(s))ds-\int_{0}^{t_{1}}S(t_{1}-s)F(X(s))ds\right)\nonumber \\  
&=& I +II, 
\end{eqnarray}
so that 
$ \Vert X(t_{2})- X(t_{1})\Vert \leq  \Vert I \Vert +  \Vert II \Vert.$
We estimate each of the terms $\Vert I\Vert $ and $\Vert  II\Vert  $. For $\Vert I\Vert $, using Proposition
\ref{prop1} yields
\begin{eqnarray*}
 \Vert I \Vert
&=&\Vert S(t_{1})A^{-1}(\text{I}-S(t_{2}-t_{1}))A^{1} X_{0} \Vert 
\quad \leq \quad   C (t_{2}-t_{1}) \Vert X_{0} \Vert_{2}.
\end{eqnarray*}

For the term $II$, we have 
\begin{eqnarray*}
II&= &\int_{0}^{t_{1}}(S(t_{2}-s)-S(t_{1}-s))F(X(s))ds +\int_{t_{1}}^{t_{2}}S(t_{2}-s)F(X(s))ds\\
  &=& II_{1}+II_{2},
 \end{eqnarray*}
with
\begin{eqnarray*}
 \Vert II\Vert \leq \Vert II_{1} \Vert+ \Vert II_{2} \Vert.
\end{eqnarray*}

We now estimate each term $ \Vert II_1\Vert $ and $\Vert  II_2\Vert $. For $ \Vert II_1 \Vert$ 
\begin{eqnarray*}
 \Vert II_{1} \Vert &=& \Vert
 \int_{0}^{t_{1}}(S(t_{2}-s)-S(t_{1}-s))F(X(s))ds \Vert\\ 
  &\leq&  \int_{0}^{t_{1}}\Vert
    (S(t_{2}-s)-S(t_{1}-s)) F(X(s))\Vert ds \\ 
  &\leq&  \left(\int_{0}^{t_{1}}\Vert (S(t_{2}-s)-S(t_{1}-s))\Vert_{L(H)} 
    ds\right) \,\left(\underset{0\leq s\leq T}{\sup} \Vert F(X(s))\Vert\right).
\end{eqnarray*}
For $\epsilon \in (0,1/4)$ small enough, using Proposition \ref{prop1}  yields
\begin{eqnarray*}
   \Vert II_{1} \Vert &\leq&
  \left(\int_{0}^{t_{1}}\Vert S(t_{1}-s) A^{1-\epsilon} A^{-1+ \epsilon}(\text{I}-S(t_{2}-t_{1}))\Vert_{L(H)}  ds\right)
  \, \left(\underset{0\leq s\leq T}{\sup} \Vert F(X(s))\Vert\right)\\ 
&\leq&  \left(\int_{0}^{t_{1}}\Vert A^{1-\epsilon} S(t_{1}-s)
  A^{-1+\epsilon}(\text{I}-S(t_{2}-t_{1}))\Vert_{L(H)}  ds\right)
\, \left(\underset{0\leq s\leq T}{\sup} \Vert F(X(s))\Vert\right)\\ 
&\leq& C (t_{2}-t_{1})^{1-\epsilon}
\left(\int_{0}^{t_{1}}(t_{1}-s)^{-1+\epsilon} ds\right )\,
\left(\underset{0\leq s\leq T}{\sup} \Vert F(X(s))\Vert\right)\\ 
&\leq& C (t_{2}-t_{1})^{1-\epsilon}\, \left(\underset{0\leq s\leq T}{\sup} \Vert F(X(s))\Vert\right). 
\end{eqnarray*}

For $\Vert II_2 \Vert $, using  the fact that the semigroup is bounded, we have 
\begin{eqnarray*}
 \Vert II_2 \Vert &=& \Vert \int_{t_{1}}^{t_{2}}S(t_{2}-s)F(X(s))ds\Vert\\
&\leq&  \left(\int_{t_{1}}^{t_{2}}\Vert S(t_{2}-s)F(X(s)) \Vert ds \right)\\
&\leq& \left(\int_{t_{1}}^{t_{2}}\Vert F(X(s))\Vert ds \right)\\
 &\leq & C (t_{2}-t_{1}) \left(\underset{0\leq s\leq T}{\sup} \Vert F(X(s))\Vert\right). 
\end{eqnarray*}
Hence
\begin{eqnarray*}
 \Vert II\Vert \leq  \Vert II_{1} \Vert +\Vert II_{2} \Vert )\leq C (t_{2}-t_{1})^{1-\epsilon}  \left(\underset{0\leq s\leq T}{\sup} (\Vert F(X(s))\Vert) \right).
\end{eqnarray*}
Using the fact that $F$ satisfies \eqref{lip1}, we therefore have
\begin{eqnarray}
\label{conseq}
 \Vert F(X(t)) \Vert &\leq&   \Vert F (X_0)\Vert + \Vert F (X(t)) - F (X_0)\Vert \nonumber\\
 &\leq& \Vert F (X_0)\Vert + L \Vert X(t) -X_0\Vert_1  \nonumber\\
 &\leq& C( \mathcal{B},X_0,F).
\end{eqnarray}
Combining  \eqref{conseq} and previous estimations of $\Vert I\Vert $ and  $\Vert II\Vert $ ends the proof of part (i).

\textbf{Proof of part (ii)}. 
We consider again the difference in \eqref{ddif}.
The difference with the proof of part (i) comes from the estimation of $II_{1}$. This time we write
\begin{eqnarray*}
  II_{1}&=&\int_{0}^{t_{1}}(S(t_{2}-s)-S(t_{1}-s))F(X(s))ds\\
  &=& \int_{0}^{t_{1}}(S(t_{2}-s)-S(t_{1}-s))
  \left( F(X(s))-F(X(t_{1}))\right)ds \\ 
  &+& \int_{0}^{t_{1}}(S(t_{2}-s)-S(t_{1}-s))F(X(t_{1}))ds\\
  &=& II_{11}+ II_{12}.
\end{eqnarray*}
Remember  that $\Vert. \Vert_{-1}= \Vert A^{-1/2}. \Vert $ in $V^*$, since $H^{-1}(\Omega)=V^*= \mathcal{D}(A^{-1/2})$ as $V= \mathcal{D}(A^{1/2})$.   If $F$ satisfies the Lipschitz condition  given in \eqref{lip}, then
using the result in part (i) together with \propref{prop1} yields   
\begin{eqnarray*}
 \Vert II_{11} \Vert  &\leq&\left(\int_{0}^{t_{1}}\Vert \left(
   S(t_{2}-s)-S(t_{1}-s)\right)A^{1/2}\Vert_{L(H)} \Vert A^{-1/2} \left( F(X(s))-F(X(t_{1})\right)\Vert
   ds\right)\\ 
&\leq& C(\mathcal{B})\left(\int_{0}^{t_{1}}\Vert
  A^{1/2}\left( S(t_{2}-s)-S(t_{1}-s)\right)\Vert_{L(H)}  \Vert X(s)-X(t_{1})\Vert
   ds\right)\\ 
      &\leq& C(\mathcal{B})\left(\int_{0}^{t_{1}}\Vert
  A^{3/2}S(t_1-s) A^{-1}\left(I-S(t_{2}-t_1)\right)\Vert_{L(H)}  \Vert X(s)-X(t_{1})\Vert
   ds\right)\\
 &\leq& C(\mathcal{B}) \left((t_{2}-t_{1}) \int_{0}^{t_{1}} (t_{1}-s)^{-\epsilon-1/2}
   ds\right)\\ 
 &\leq & C(\mathcal{B}) \left(t_{2}-t_{1}\right).
\end{eqnarray*}
 We also  have 
 \begin{eqnarray*}
   \Vert II_{12} \Vert  &\leq&
   \Vert F(X(t_{1})\Vert
   \Vert\int_{0}^{t_{1}}( S(t_{2}-s)-S(t_{1}-s)) ds\Vert_{L(H)}\\ 
   &\leq& C(\mathcal{B}) \Vert \int_{0}^{t_{1}} S(t_{2}-s)-S(t_{1}-s)ds \Vert_{L(H)}. 
\end{eqnarray*}
 Using the two transformations $y=t_{2}-s,\; y=t_{1}-s$, we find
 \begin{eqnarray*}
   \Vert II_{12} \Vert 
   &= & C(\mathcal{B}) \Vert\int_{t_{2}- t_{1}}^{t_{2}} S(y) dy - \int_{0}^{t_{1}}
   S(y) dy \Vert_{L(L^{2}(\Omega))} \\ 
   &= & C(\mathcal{B}) \Vert\int_{t_{2}- t_{1}}^{t_{1}} S(y) dy +
   \int_{t_{1}}^{t_{2}} S(y) dy - \int_{0}^{t_{1}} S(y)dy \Vert_{L(H)} \\ 
   &= & C(\mathcal{B}) \Vert  \int_{t_{1}}^{t_{2}} S(y) dy - \int_{0}^{t_{2}-t_{1}}
   S(y)dy \Vert_{L(H)} \\ 
   &\leq& C(\mathcal{B}) (t_{2}-t_{1}).
\end{eqnarray*}

The estimate of $II_{1}$ combined with \eqref{conseq} in the  estimate of  $II_{2}$ ends the proof.
\end{proof}
\section{Finite Element method  for semi-linear parabolic problem}
Finite element method for space discretization has been used  in \cite{Stig} for semilinear problem \eqref{adr}, time discretization has been performed
using first order implicit and semi-implicit methods under the  
locally Lipschitz condition \eqref{lipf} or  \eqref{lipff}.
The convergence proofs of  parabolic stochastic partial
differential equations with multiplicative or additive noise are provided in \cite{GtambueIma,GTambue,GTambueexpo} 
where the space and time discretizations are performed using respectively finite element method and exponential integrators schemes. 
The convergence proof for deterministic  problem \eqref{adr} using locally Lipschitz condition \eqref{lipf} can easily be deducted by canceling the noise term in \cite{GtambueIma,GTambue,GTambueexpo} and combined with different results in \cite{Stig}.
 
 The keys features while using finite element method for space discretization comes from  the fact that the corresponding semi-discrete problem  shares  the same bilinear 
 form \eqref{tvar} with the continuous problem \eqref{adr}.
The convergence proof for deterministic  problem \eqref{adr} using locally Lipschitz condition \eqref{lipf} or \eqref{lipff} will be more difficult with finite volume method (or finite difference method) 
 for space discretization since the corresponding  bilinear form of the semi-discrete problem \eqref{discretevar} is different  with the one of  the continuous problem.
 \section{Finite volume for space discretization}
\label{sec:FVm}
\subsection{Admissible  mesh}
A cell--centred finite volume methods for heterogeneous and
 anisotropic diffusion problems remains a challenging problem.
 An active area of research consists to make the approximation of the 
diffusion flux more efficient and simple as possible (see \cite{FV1} for the references).
The finite volume method is widely  applied when the 
differential equations are in divergence form. To obtain a 
finite volume discretization, the domain $\Omega$ is  subdivided 
into subdomains $(A_{i})_{i \in \mathcal{I}},\;\mathcal{I}$ being the corresponding set of indices, called control volumes 
or control domains such that the collection of all those 
subdomains forms a partition of  $\Omega$. 
The common feature of all finite volume methods is to integrate 
the equation over each control volume $A_{i},\;i\in \mathcal{I}$  
and apply Gauss's divergence theorem to convert the volume integral
 to a surface integral. For our parabolic problem \eqref{adr6},
 finite volume methods differ in the way they approximate the
 diffusion flux $\mathcal{F}=-\mathbf{D}\nabla X$. Two techniques are mostly used: 
 the finite volume with  
two-point flux approximation (TPFA) (see \cite{FV,FV1}) and the finite volume with 
multi-point flux approximations (MPFA)(\cite{MPFA,MPFA1}). 
 
An advantage of the two-point approximation is that it provides
 monotonicity properties, under the form of a local maximum principle. It is efficient
 and mostly used in industrial simulations. 
In this paper we use the TPFA  as developed in \cite{FV}.
The main  drawback of TPFA is that it is   applicable in the
so called `` admissible mesh ''  or ''$\mathbf{D}$-orthogonal mesh'' and not in general mesh.

\begin{definition} 
\label{admissiblemesh}
[\textbf{Admissible  mesh}]

An  admissible mesh $\mathcal{T}$ for  problem \eqref{adr6} with the full diffusion tensor $\mathbf{D}$ is defined by: 
\begin{itemize}
 \item  A set $ \left\lbrace   A_{i} \right\rbrace_{i \in \mathcal{I}}$  of control volumes such that $\overline{\Omega}= \underset{i \in \mathcal{I}}{\cup}\overline{A_{i}} $ 
with the corresponding local inner product induced by $\mathbf{D}_{A_{i}}^{-1}$ where
$$ \mathbf{D}_{A_{i}}=\dfrac{1}{ \mathrm{mes}(A_{i})}\int_{A_{i}}\mathbf{D}(\mathbf{x})d\mathbf{x}.$$
 \item The corresponding  set of center points  $\left\lbrace \mathbf{x}_{i} \right\rbrace_{i \in \mathcal{I}} $ such that 
\begin{enumerate}
\item[(a)] $\mathbf{x}_{i} \in \overline{A}_{i},\; i \in \mathcal{I}$.  
\item [(b)]  $\mathbf{x}_{i}$ is the intersection of the straight lines perpendicular to the boundary of 
$A_{i}$ with respect to the inner  product induced by $\mathbf{D}_{A_{i}}^{-1}$.
\end{enumerate}
\end{itemize}
\end{definition}
Let $h=\text{size} (\mathcal{T})$ be the maximum mesh size of $\mathcal{T}$. We denote by $\mathcal{T}_{h}$ a dual Delaunay triangulation of $\mathcal{T}$
i.e. a Delaunay triangulation  where $\left\lbrace \mathbf{x}_{i} \right\rbrace_{i \in \mathcal{I}}$ 
is the set of vertices (2-D delaunay Triangulation with triangular mesh  or 3-D delaunay Triangulation with tetrahedal mesh).  
For a given  set $\left\lbrace \mathbf{x}_{i} \right\rbrace_{i \in \mathcal{I}}$, a dual mesh  $\mathcal{T}_{h}$  can easily be constructed with the Matlab  function \textbf {delaunayTriangulation.m}.

\begin{figure}[!ht]
\begin{center}
  \includegraphics[width=0.4\textwidth]{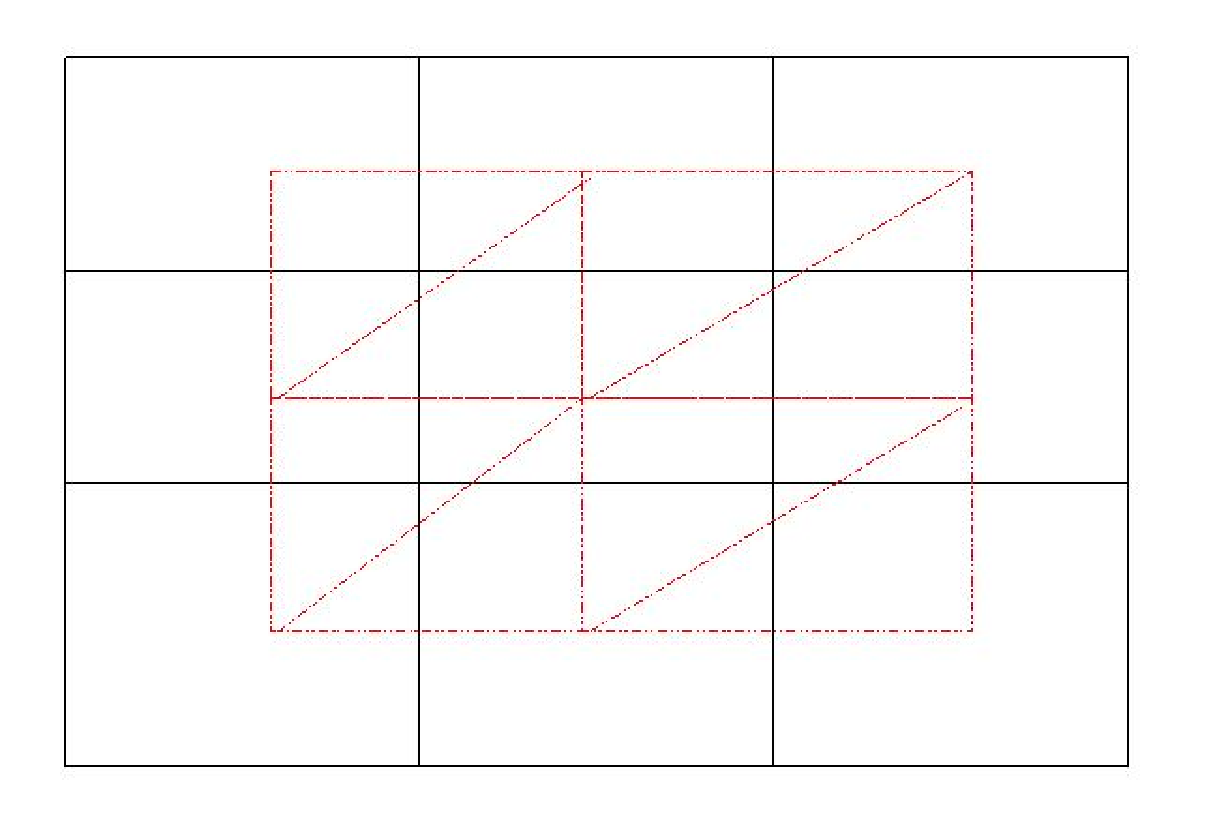}
  \end{center}
  \caption{Example of admissible mesh in $d=2$ for diagonal diffusion tensor $\mathbf{D}$. The mesh  $\mathcal {T}$ is the set of rectangular grid (in black line ) and  a corresponding  dual Delaunay triangulation
  $\mathcal{T}_{h}$(triangle mesh) is in red line. For $d=3$, this figure can represent the upper view of the set of parallelepiped grid with the corresponding dual Delaunay triangulation (tetrahedal mesh). Note that the dual triangulation  $\mathcal{T}_{h}$
  is not unique here.}
  \label{mesh}
\end{figure}
Let us illustre \defref{admissiblemesh} to make it more  understandable.
\begin{example}
\begin{itemize}
\item  In the case where  the diffusion tensor $\mathbf{D}$ is diagonal and  $\Omega$ is a rectangular or parallelepiped domain,
 any rectangular grid ($d=2$) or parallelepiped grid ($d=3$) is an admissible mesh. The set $\left\lbrace\mathbf{x}_{i}\right\rbrace$ 
is the set of centers of gravity of the rectangular grid or parallelepiped grid (see \figref{mesh}). The inner product induced locally by  $\mathbf{D}_{A_{i}}^{-1}$ 
is equivalent to the standard  inner product corresponding to the Euclidean norm $\vert.\vert$.
This mesh will yield a 5-point scheme ($d=2$) and  7-point scheme ($d=3$) for our model problem \eqref{adr6}.
 \item If $d=2$, for isotropic and heterogeneous media ($\mathbf{D}(\mathbf{x})=b(\mathbf{x})I_{2}\;\,\mathbf{x} 
\in \Omega $, $I_{2}$ being the identity matrix of dimension $2$) we can define a triangular admissible
 mesh $\mathcal {T}$ 
   to be a family of open triangular disjoint subsets of $\Omega$ such that two triangles having a 
common edge have also two common vertices. The angles of the triangles are assumed to be less than $\frac{\pi}{2}$ to 
allow the orthogonal bisectors to intersect
 inside each triangle, thus naturally defining the center point $\mathbf{x}_{i}$ of the control volume $A_{i}$.
The finite volume scheme defined on such  mesh  will yield a 4-point scheme for our model problem \eqref{adr6}. 
The inner product induced locally by  $\mathbf{D}_{A_{i}}^{-1}$ 
is equivalence to the standard  inner product corresponding to the Euclidean norm $\vert.\vert$.
\end{itemize}
\end{example}
\subsection{Finite volume space discretization and  semi-discrete solution}
Consider the modified model problem of (\ref{adr6}) where $c_{0}X$  is added on both sides  of the first equation
 of problem (\ref{adr6}), $c_{0}$ is defined in (\ref{coer}).
Consider an admissible mesh $\mathcal{T}$ in the sense of \defref{admissiblemesh}.
Denote by $\mathcal{E}$ the set of edges of control volume of $ \mathcal{T},\;\mathcal{E}_{int} $ 
the set of interior  edges of control volume of $ \mathcal{T}$, $X_{i}(t)$ 
the approximation of $X$ at time $t$ at the center (or at any point) of the control 
volume $A_i\in \mathcal{T}$ and $X_{\sigma}(t)$ the approximation of $X$ at time $t$ at 
the center (or at any point)of the edge $\sigma \in \mathcal{E}$.
 For a control volume $A_i\in \mathcal{T}$, denote by $\mathcal{E}_{i}$ the set
 of edges of $A_i$, 
$\mathrm{mes}(A_i)$ the Lebesgue measure of the control
 volume $A_i\in \mathcal{T}$. 

As in \cite{FV,EP}, integration over any control volume $A_i\in \mathcal{T}$, 
using the divergence theorem 
to convert the integral over $A_i$ to a surface integral, 
finite differences for the diffusion flux approximation \cite{FV} and 
the upwind technique 
 for the advection flux  approximation  yields
\begin{eqnarray}
\label{ode}
\left\lbrace \begin{array}{l}
 \mathrm{mes}(A_i)\dfrac{d X_{i}(t)}{dt}+\underset{\sigma \in \mathcal{E}_{i}}{\sum}\left( F_{i,\sigma}(t)+ q_{i,\sigma} X_{\sigma,+}(t)\right) + c_{0}\,\mathrm{mes}(A_i)X_{i}(t)\\
 \qquad \qquad \qquad \qquad\qquad \qquad \qquad \qquad\qquad \qquad
= \mathrm{mes}(A_i)\,f(\mathbf{x}_{i},X_{i}(t)), \\
\newline\\
D_{i,\sigma}=\vert \mathbf{D}_{A_i}\,\mathbf{n}_{i,\sigma}\vert,\quad\mathbf{D}_{A_i}=\dfrac{1}{ \mathrm{mes}(A_i)}\int_{A_i}\mathbf{D}(\mathbf{x})d\mathbf{x},\\
\newline\\
F_{i,\sigma}(t)= \mathrm{mes}(\sigma)\;D_{i,\sigma}\dfrac{X_{\sigma}(t)-X_{i}(t)}{d_{i,\sigma}}, \\
\newline\\
q_{i,\sigma}=\int_{\sigma}\mathbf{q} \cdot \mathbf{n}_{i,\sigma}d\sigma \qquad \qquad \qquad \qquad  \quad\qquad  \quad\qquad  \forall A_i \in\; \mathcal{T},\;\;\forall \sigma \in \mathcal{E}_{i}.
\end{array}\right.
\end{eqnarray}
 Here \;$\mathbf{n}_{i,\sigma}$ is the  normal unit vector to $\sigma $  outward to  $A_i$,
 $\mathrm{mes}(\sigma)$ is the Lebesgue measure 
of the edge $\sigma \in \mathcal{E}_{i}$ and $d_{i,\sigma}$ the distance between the center of $A_i$ and the edge $\sigma$.

 Since the flux is continuous at the interface of two control volumes $A_i$ and $A_j$ (denoted by $i\mid j$) we therefore have $F_{i,\sigma}(t)= -F_{j,\sigma}(t)$ for $ \sigma=i\mid j$, which  yields 
\begin{eqnarray}
\label{interf}
\left\lbrace \begin{array}{l}
F_{i,\sigma}(t)=-\tau_{\sigma}\left( X_{j}(t)-X_{i}(t)\right)= - \dfrac{\mu_{\sigma}\, \mathrm{mes}(\sigma)}{d_{i,j}}\left( X_{j}(t)-X_{i}(t)\right),\, \sigma=i\mid j \quad 
\newline\\
\tau_{\sigma}=\mathrm{mes}(\sigma) \dfrac{D_{i,\sigma}D_{j,\sigma}}{D_{i,\sigma}\; d_{i,\sigma}+D_{j,\sigma}
 d_{j,\sigma}} \quad  (\text{transmissibility through}\, \sigma)\\
\end{array}\right.
\end{eqnarray}
with 
\begin{eqnarray}
\label{transn1}
 \mu_{\sigma}=
 d_{i,j} \dfrac{D_{i,\sigma}D_{j,\sigma}}{D_{i,\sigma}\; d_{i,\sigma}+D_{j,\sigma}d_{j,\sigma}},
\end{eqnarray}
where $d_{i,j}$ is the distance between the center of $A_i$ and center of $A_j$. 
We will set  $d_{i,j}=d_{i,\sigma}$  for $\sigma=\mathcal{E}_{i} \cap \partial\Omega $. For $\sigma \subset \partial \Omega$, we  also write
\begin{eqnarray*}
\label{ninterf}
F_{i,\sigma}(t)
&=&-\tau_{\sigma}\left( X_{j}(t)-X_{i}(t)\right)\\
              &=& -\dfrac{\mathrm{mes}(\sigma) \mu_{\sigma}}{d_{i,\sigma}}\left( X_{j}(t)-X_{i}(t)\right).
\end{eqnarray*}
with
\begin{eqnarray}
\label{trans2}
\left\lbrace \begin{array}{l}
X_{j}(t)= X_{\sigma}(t)=0\\
\newline\\
 \tau_{\sigma} =\dfrac{\mathrm{mes}(\sigma) D_{i,\sigma}}{d_{i,\sigma}}  \\
\newline\\
\label{transn11}
 \mu_{\sigma}= D_{i,\sigma}.
\end{array}\right.
\end{eqnarray}

The upwind term  for advection flux $X_{\sigma,+}$ is defined as
\begin{eqnarray}
 X_{\sigma,+}(t)&=&\left\lbrace \begin{array}{l}
  X_{i}( t)\quad \text{if} \quad q_{i,\sigma}\geqslant 0  \\  
       \newline\\             
   X_{j}( t) \quad  \text{if} \quad  q_{i,\sigma}< 0         
\end{array}\right.
\;\; \;\;\text{for}\;\;\;\sigma = i\mid j\\ 
X_{\sigma,+}(t)&=&\left\lbrace \begin{array}{l}
    X_{i}( t)\quad \text{if} \quad  q_{i,\sigma}\geqslant 0 \\
\newline\\
    X_{\sigma}(t) \quad  \text{if}  \quad  q_{i,\sigma}< 0                    
                             \end{array}\right.
\;\; \text{for}\;\;\sigma \in\mathcal{E}_{i}\cap \partial \Omega.
\end{eqnarray}
We can write $X_{\sigma,+}$ as
\begin{eqnarray}
\label{up}
  X_{\sigma,+}= r_{\sigma} X_{i}(t)+(1-r_{\sigma})X_{j}(t),\quad \sigma=i\mid j  
\end{eqnarray}
where $r_{\sigma}= \dfrac{1}{2}(\text{sign}(q_{i,\sigma})+1)$.
Note that according to \eqref{N}, we have 
\begin{eqnarray}
 (F(X_i))(\mathbf{x}_{i})=f(\mathbf{x}_{i},X_{i}(t)).
\end{eqnarray}

Using previous approximations, the finite volume space discretization  for the model problem (\ref{adr6}) is given by
\begin{eqnarray}
\label{fvfd}
\left\lbrace \begin{array}{l}
 \mathrm{mes}(A_i)\dfrac{d X_{i}(t)}{dt}+\underset{\sigma\in \mathcal{E}_{i}}{\sum}\left( -\dfrac{ \mathrm{mes}(\sigma)\,\mu_{\sigma}}{d_{i,j}}\left( X_{j}(t)-X_{i}(t)\right) \right. \\
\newline\\
\left. + q_{i,\sigma} \left(r_{\sigma} X_{i}(t)+(1-r_{\sigma})X_{j}(t)\right)\right) +c_{0}\,\mathrm{mes}(A_i)X_{i}(t) = \mathrm{mes}(Ai) F(X_{i}(t)) \\
\newline\\
 X_{i}(t)=0, \: d_{i,j}=d_{i,\sigma}  \qquad \qquad \qquad  \qquad \qquad \qquad   \text{if}\; \sigma \subset \partial \Omega, \quad \forall A_i \;\in \mathcal{T}.
\end{array}\right.
\end{eqnarray}

The scheme (\ref{fvfd}) clearly indicates the affinity of the finite volume method to the
 finite difference method. However, for the subsequent analysis it is more convenient
 to rewrite scheme (\ref{fvfd}) in a discrete variational form.

Multiplying the first equation of (\ref{fvfd}) by  arbitrary numbers $v_{i}\in \mathbb{R}$ 
and summing the results over all control volume in $\mathcal{T}$ yields
\begin{eqnarray}
\label{fvfd1}
\left\lbrace \begin{array}{l}
 \underset{A_i \in \mathcal{T}}{\sum} \left[\mathrm{mes}(A_i)\dfrac{d X_{i}(t)}{dt}+\underset{\sigma \in \mathcal{E}_{i}}{\sum}\left(\dfrac{\mathrm{mes}(\sigma)\,\mu_{\sigma}}{d_{i,j}}\left( X_{i}(t)-X_{j}(t)\right)\right.\right. \\
\newline \\
\left.\left. +q_{i,\sigma} \left(r_{\sigma} X_{i}(t)+(1-r_{\sigma})X_{j}(t)\right)\right)\right] v_{i}=  \underset{A_i \in \mathcal{T}}{\sum} \mathrm{mes}(A_i)\,F(X_{i}(t))v_{i}.
\end{array}\right.
\end{eqnarray}
Let $V_{h}$ denote the space of continuous functions that are piecewise linear over the Delaunay triangulation $\mathcal{T}_{h}$ (dual of $\mathcal{T}$),
 and $X(\mathcal{T})$ be the space of the functions constant in each control volume of $\mathcal{T}$,  the following  lemma creates a one-to-one correspondence between  $V_h$ and $X(\mathcal{T})$.
  \begin{lemma}
  \label{Vh}
   There exists a one-to-one correspondence between  the space $V_h$ and $X(\mathcal{T})$, more precisely  we have :
   \begin{itemize}
   \item For any $U_h \in V_h$ corresponds  the unique function $U=(U_h(\mathbf{x}_{i}))_{i \in \mathcal{I}} \in X(\mathcal{T})$.
    \item  For any  function $U=(U_i)_{i \in \mathcal{I}} \in X(\mathcal{T})$, there exists an unique function  $U_h \in V_h$ such  that $U_h(\mathbf{x}_{i})=U_i$.
   \end{itemize}
 
    \end{lemma}

 \begin{proof}
  The first correspondence is obvious. The proof of the second correspondence can be found in  \cite[Lemma 2.10,  p.58]{EP}.
 \end{proof}

Let us consider equation \eqref{fvfd1},  according to \lemref{Vh}, there are unique functions $X_{h}(t),v_{h} \in V_{h}$ such that $X_{h}(t)(\mathbf{x}_{i})=X_{i}(t)$ and $ v_{h}(\mathbf{x}_{i})=v_{i} $ for all $A_i \in \mathcal{T}$, 
where $\mathbf{x}_{i}$ is a center of the control volume $A_i \in \mathcal{T}$ ($\mathbf{x}_{i}$ is also a vertex in $\mathcal{T}_{h}$).

Denote by $a_{h}$ the bilinear form defined by
\begin{eqnarray}
\label{discretevar}
\left\lbrace \begin{array}{l}
 a_{h}(u_{h},v_{h})= \underset{A_i \in \mathcal{T}}{\sum} \underset{\sigma \in \mathcal{E}_{i}}{\sum} \left(- \dfrac{ \mathrm{mes}(\sigma)\,\mu_{\sigma}}{d_{i,j}}\left( u_{j}-u_{i}\right) + q_{i,\sigma} \left(r_{\sigma} u_{i}+(1-r_{\sigma})u_{j}\right) \right)v_{i}\\
\newline \\
\qquad\qquad +c_{0}\,\mathrm{mes}(A_i)u_{i} v_{i}\qquad \qquad \qquad \qquad \qquad \qquad \qquad \forall u_{h},v_{h} \in V_{h},
\end{array}\right.
\end{eqnarray}
and  by $\langle.,.\rangle_{0,h}$ the scalar product on $ C(\overline{\Omega})\supset V_{h}$ defined by
\begin{eqnarray}
\label{dproduct}
 \langle u,v\rangle_{0,h}= \underset{i \in \mathcal{T}}{\sum} \mathrm{mes}(A_i)u_{i}v_{i},\quad u_{i}=u(\mathbf{x}_{i}),\quad v_{i}=v(\mathbf{x}_{i}),\quad \quad u,v \in C(\overline{\Omega}).
\end{eqnarray}
Note  that this scalar  product can be extended in $L^2(\Omega)$ (see \cite{florin}) by 
\begin{eqnarray}
\label{dproductL2}
\langle u,v\rangle_{0,h}= \underset{i \in \mathcal{T}}{\sum} \mathrm{mes}(A_i)u_{i}v_{i},
\quad u_{i}= \dfrac{1}{\mathrm{mes}(A_i)}\int_{A_i} u dx, \quad v_{i}=\dfrac{1}{\mathrm{mes}(A_i)}\int_{A_i} v dx ,\quad u,v \in L^2(\Omega).
\end{eqnarray}
Note also that when $ u\in C(\overline{\Omega})$ we will use $u_{i}=u(\mathbf{x}_{i})$  in \eqref{dproductL2}.
We can easily observe for  $ u, v \in X(\mathcal{T})$ \footnote{Remember that this is the space of the functions constant in each control volume of $\mathcal{T}$}
$\langle u,v\rangle_{0,h}=(u,v)$. The corresponding norm  of $\langle.,.\rangle_{0,h}$ is the discrete $L^{^{2}}(\Omega)$ norm  denoted by $\Vert.\Vert_{0,h}$. 
We therefore have the following  variational form of our finite volume scheme (\ref{fvfd1}).
\begin{eqnarray}
\label{varia1}
\left\lbrace \begin{array}{l}
 \langle \dfrac{d}{dt}X_{h},\varphi \rangle_{0,h}+a_{h}(X_{h}(t),\varphi)= \langle F(X_{h}(t)),\varphi \rangle_{0,h},\quad \forall \varphi \in V_{h},\quad t \in \left( 0, T\right],\\
\newline\\
X_{h}(0)=X_{h\,0}.\\
\end{array}\right.
\end{eqnarray}
Consider the operator $A_{h}: V_{h}\rightarrow V_{h} $ such that
\begin{eqnarray}
\langle A_{h} \psi,\chi \rangle_{0,h} =a_{h}(\psi,\chi)\;\;\;\; \forall \psi,\chi \in  V_{h}.
\end{eqnarray}
The semidiscrete solution in $V_{h}$ is then given by: find $X_{h}(t) \in V_{h}$ such that
\begin{eqnarray}
\label{adrh}
\left\lbrace \begin{array}{l}
 \dfrac{d X_{h}}{dt}+ A_{h}X_{h}=P_{h}F(X_{h})\quad \;\;t\in \left(0,T\right] \\
\newline\\
X_{h}(0)=X_{0 h}
\end{array}\right.
\end{eqnarray}
where $P_{h}$ is the orthogonal projection defined from $ L^2(\Omega)$ to $V_{h}$ by
\begin{eqnarray}
\label{fvprojection}
\langle P_{h}u,\chi \rangle_{0,h}= \langle u,\chi \rangle_{0,h}\;\;\;\forall \chi \in V_{h},\,\, u \in L^2(\Omega).
\end{eqnarray}
In order to provide the corresponding mild form of \eqref{adrh}, let us define the  discrete $H_{0}^{1}(\Omega)$ norm. 
\begin{definition}
   \textbf{[Discrete $H_{0}^{1}(\Omega)$ norm \cite{FV}]}\\
\label{seminorm}
 Let  $\mathcal{T}$ be an admissible finite volume mesh in the sense of \defref{admissiblemesh}, and  $X(\mathcal{T})$
 the space of the functions constant in each control volume of $\mathcal{T}$.
For $u \in X(\mathcal{T})$ corresponding to $u_h \in V_h$ (according to \lemref{Vh}), the  discrete  $H_{0}^{1}(\Omega)$ norm of $u$ and $u_h$ is defined by 
 \begin{eqnarray}
  \Vert u_h \Vert_{1,\mathcal{T}}: =\Vert u \Vert_{1,\mathcal{T}}:= \left( \underset{\sigma \in \mathcal{E}}{\sum} \tau_{\sigma}'\left(D_{\sigma}u\right)^{2}\right)^{1/2}
 \end{eqnarray}
 where 
  \begin{eqnarray*}
   \tau_{\sigma}'&=&\dfrac{mes(\sigma)}{d_{\sigma}}\\
  D_{\sigma}u &=&\vert u_{i}-u_{j}\vert \qquad \text{if}\qquad \sigma= i\vert j \in \mathcal{E}_{int}\\
 D_{\sigma}u &=&\vert u_{i}\vert \qquad  \qquad \text{if}\qquad  \sigma \in \partial \Omega.
\end{eqnarray*}
   \end{definition}
   Note that  this norm is equivalent in $V=H_{0}^{1}(\Omega)=\mathcal{D}(A^{1/2})=\mathcal{D}(A^{* 1/2})$ to  the  natural norm of $H^1(\Omega)$ which is $ \Vert . \Vert_1$.
   Following closely \cite{Stig,Lionsj} we also have $\mathcal{D}(A_h^{1/2})=\mathcal{D}(A_h^{* 1/2})$ with the following norm equivalence
   \begin{eqnarray*}
   \label{equivd}
    \Vert A_h^{1/2} u\Vert \equiv \Vert A_h^{*1/2} u\Vert \equiv \Vert u\Vert_1, \qquad \qquad  u \in V_h. 
   \end{eqnarray*}

   We make the following assumption as in \cite[Theorem 3.8]{FV}, very useful for our convergence proof.
\begin{Assumption}
\label{meshassumption}
[\textbf{Regularity of $\mathbf{D}$, $\mathbf{q}$ and $\mathcal{T}$}]\\
We assume that $\mathbf{D}$ is bounded\footnote{From \eqref{ellipticity}, $\mathbf{D}$ is bounded below, to be bounded $\mathbf{D}$ also need to be bounded above.}, $\text{the restriction of}\;\; \mathbf{D}\; \text{to  any} \;A_i \in \mathcal{T} \;\text{belongs to}
\;C^{1}(A_i,\mathbb{R}^{d\times d}),\;q_{j}\in C^{1}(\overline{\Omega})$, the discontinuities of $\mathbf{D}$ coincide with the interfaces of the mesh,
and that there exists $\zeta_{1}>0$ such that
\begin{eqnarray}
\label{meshassumption1}
\zeta_{1} h\leq d_{i,\sigma},\quad \quad \forall\,\, A_i \in \mathcal{T},\,\, \ \forall\, \sigma \in \mathcal{E}_i,
\end{eqnarray}
where $h=\text{size}(\mathcal{T})$. 
\end{Assumption}
The inequality \eqref{meshassumption1} is called regularity property of the mesh  $\mathcal{T}$.
\begin{remark}
\label{regularduall}
 The regularity property  of the dual mesh $\mathcal{T}_h$
 given in \cite[Definition 3.28, p 138] {EP}
states that there exists some constant $c>0$ such that 
\begin{eqnarray}
\label{regulardual}
 \dfrac{h_K}{\rho_K} \leqslant c,\,\quad \quad \quad \forall K \in \mathcal{T}_h 
\end{eqnarray}
 where
$h_K=\text{diam}(K)=\underset{(x,y)\in K^{2}}{\sup}d(x,y) $ and $\rho_K= \sup \{\text{diam}(S) \mid  S \;\text{is a ball in } \mathbb{R}^{d}\; \text{and }\; S \subset K \}$.

As we are dealing in $\mathcal{T}_h$ with triangle or tetrahedron, $h_K$ denotes the longest edge and $\rho_K$ the
diameter of the inscribed circle ($d=2$) or sphere ($d=3$). Using Heron's formula  and its consequences, one can prove that the regularity of mesh  $\mathcal{T}$  given by \eqref{meshassumption1} 
implies the regularity of the dual mesh $\mathcal{T}_h$ given by \eqref{regulardual}.

 \end{remark}

\assref{meshassumption} allows  the following $ V_{h}-$ ellipticity of $a_{h}$.  
\begin{theorem}
\label{coert}
 Under the regularity of the admissible mesh $\mathcal{T}$ in \assref{meshassumption}, there exists a constant $\alpha>0$ independent of $h$ such that 
 \begin{eqnarray}
\label{dellip1}
a_{h}(v_{h},v_{h})\geq \alpha\; \Vert v_{_{h}}\Vert_{1,\mathcal{T}}^{2}\;\;\;\;\; \forall v_{h} \in V_{h}.
\end{eqnarray}
\end{theorem}
\begin{proof}
Let $b_{h}^{1},\,b_{h}^{2}$ and $b_{h}^{3}$ the bilinear forms defined  in $V_{h}\times V_{h}$ by
\begin{eqnarray}
 b_{h}^{1}(u_{h},v_{h})&=&\underset{A_i \in \mathcal{T}}{\sum} \underset{\sigma \in \mathcal{E}_{i}}{\sum} - \dfrac{ \mathrm{mes}(\sigma)\,\mu_{\sigma}}{d_{i,j}}\left( u_{j}-u_{i}\right)v_{i},  
\newline\\
b_{h}^{2}(u_{h},v_{h})&=&\underset{A_i \in \mathcal{T}}{\sum} \underset{\sigma \in \mathcal{E}_{i}}{\sum}q_{i,\sigma} \left(r_{\sigma} u_{i}+(1-r_{\sigma})u_{j} \right)v_{i} 
=\underset{A_i \in \mathcal{T}}{\sum}\underset{\sigma \in \mathcal{E}_{i}}{\sum} q_{i,\sigma} u_{\sigma,+}v_{i},
\newline\\
b_{h}^{3}(u_{h},v_{h})&=&c_{0}\underset{A_i \in \mathcal{T}}{\sum}\mathrm{mes}(A_i)\,u_{i}v_{i}. 
\end{eqnarray}
Note that according to \lemref{Vh}, we have identified  $u_{h} \in V_h$ and $v_{h}\in V_h$ to their correspondent $ (u_i)_{i \in \mathcal{T}}=(u_{h}(\mathbf{x}_i))_{i \in \mathcal{T}} \in X(\mathcal{T})$  and $(v_i)_{i \in \mathcal{T}}=(v_{h}(\mathbf{x}_i)) \in X(\mathcal{T})$ in 
 the definition of $b_{h}^{1},\,b_{h}^{2}$ and $b_{h}^{3}$.
 
Using \assref{meshassumption}, mainly  the regularity of $\mathcal{T}$ ($\zeta_{1} h\leq d_{i,\sigma} \leq h$) and the fact that the coefficients of the diffusion tensor  $\mathbf{D}$ are bounded,  there exists
two constants $C_{5}(\Omega,\zeta_{1},\mathbf{D})$ and $C_{5}'(\Omega,\zeta_{1},\mathbf{D})$ such that
\begin{eqnarray}
\label{trans1}
C_{5}\leq \mu_{\sigma}=
 d_{i,j} \dfrac{D_{i,\sigma}D_{j,\sigma}}{D_{i,\sigma}\; d_{i,\sigma}+D_{j,\sigma}d_{j,\sigma}}\leq C_{5}' ,\quad \quad \sigma= i\vert j,
\end{eqnarray}
and 
\begin{eqnarray}
\label{trans11}
 C_{5}\leq \mu_{\sigma}= D_{i,\sigma} \leq \, C_{5}',\quad \quad \sigma \in \mathcal{E}_{i}\cap \partial \Omega,
\end{eqnarray}
so that
\begin{eqnarray}
 C_{5}\leq\mu_{\sigma}\,\leq \,C_{5}'\,,\quad \quad \forall \sigma \in \mathcal{E},
\end{eqnarray}
where $\mu_{\sigma}$ is defined in \eqref{transn11} and \eqref{transn1}.

Using the fact that  the  transmissibility given in (\ref{interf}) is symmetric, i.e.  $\tau_{i\vert j}=\tau_{j\vert i}$ 
and reorganizing the summation, we therefore have 
\begin{eqnarray}
\label{eq1}
 C_{5}\, \Vert v_{_{h}}\Vert_{1,\mathcal{T}}^{2} \leq b_{h}^{1}(v_{h},v_{h})\leq C_{5}'\, \Vert v_{_{h}}\Vert_{1,\mathcal{T}}^{2}. 
\end{eqnarray}
Let use some important results from \cite{FV}.
Indeed as in \cite{FV} reordering the summation over the set of edges yields
\begin{eqnarray}
\label{egaly1}
 b_{h}^{2}(v_{h},v_{h})=\underset{\sigma \in \mathcal{E}}{\sum} q_{\sigma} \left(v_{\sigma,+}-v_{\sigma,-}\right) v_{\sigma,+}
\end{eqnarray}
where
\begin{eqnarray}
 v_{\sigma,-} &=&\left\lbrace \begin{array}{l}
    v_{i}\quad \text{if} \quad  q_{i,\sigma}\leq 0 \\
\newline\\
    v_{j} \;(\text{or}\; v_{\sigma}) \quad  \text{if}  \quad  q_{i,\sigma}>0                    
                             \end{array}\right.
\;\sigma \,\in \mathcal{E}_{int} \,(\text{or}\;\sigma  \in\mathcal{E}_{i}\cap \partial \Omega),\\
q_{\sigma}&=&\vert \int_{\sigma}\mathbf{q} \cdot \mathbf{n}_{i,\sigma}d\sigma \vert.
\end{eqnarray}
Note that
\begin{eqnarray}
\label{egaly}
\underset{\sigma \in \mathcal{E}}{\sum} q_{\sigma} \left(v_{\sigma,+}-v_{\sigma,-}\right) v_{\sigma,+}= \frac{1}{2}\underset{\sigma \in \mathcal{E}}{\sum} q_{\sigma}\left(\left(v_{\sigma,+}-v_{\sigma,-}\right)^{2}+\left(v_{\sigma,+}^{2}-v_{\sigma,-}^{2}\right)\right).
\end{eqnarray}
As we have assumed  divergence-free flow, we therefore have
\begin{eqnarray}
\underset{\sigma \in \mathcal{E}}{\sum} q_{\sigma}\left(v_{\sigma,+}^{2}-v_{\sigma,-}^{2}\right)= \underset{A_i\in \mathcal{T}}{\sum} \left(\int_{A_i}\mathbf{q} \cdot \mathbf{n}_{i,\sigma}d\sigma \right) v_{i}^{2}=\int_{\Omega} \nabla \cdot \mathbf{q}(\mathbf{x})\,v_{h}^{2}\,(\mathbf{x})dx =0.
\end{eqnarray}
Then from \eqref{egaly}, we have
\begin{eqnarray}
\label{eq2}
 b_{h}^{2}(v_{h},v_{h})\geq 0.
\end{eqnarray}
We also have
\begin{eqnarray}
\label{eq3}
b_{h}^{3}(v_{h},v_{h})=c_{0}\Vert v_{h}\Vert_{0,h}\geq 0,
\end{eqnarray}
Combining \eqref{eq1},\eqref{eq2} and \eqref{eq3} yields
\begin{eqnarray}
a_{h}(v_{h},v_{h})\geq C_{5}\; \Vert v_{_{h}}\Vert_{1,\mathcal{T}}^{2}\;\;\;\;\; \forall v_{h} \in V_{h}.
\end{eqnarray}
So we should take $\alpha=C_{5}$.
\end{proof}

The following $ V_{h}-$ ellipticity of $a_{h}$ implies that  $-A_{h}$ is a sectorial on  $H=L^{2}(\Omega)$ (uniformly in $h$) i.e. there exists   $C_{1},\,\theta \in (\frac{1}{2}\pi,\pi)$, such that
\begin{eqnarray}
\label{dsecto}
\Vert (\lambda I +A_{h})^{-1} \Vert_{L(H)} \leq \dfrac{C_{1}}{\vert \lambda \vert },\qquad \qquad \qquad
\lambda \in S_{\theta},
\end{eqnarray}
where $S_{\theta}=\left\lbrace  \lambda \in \mathbb{C} :  \lambda=\rho e^{i \phi},\; \rho>0,\;0\leq \vert \phi\vert \leq \theta \right\rbrace $.

 The discrete operator  $-A_{h}$  therefore is  the infinitesimal generator of bounded analytic semigroup (or exponential operator)  $S_{h}(t):= e^{-t\,A_{h}}$ on $V_{h}$ such that
\begin{eqnarray}
\label{disretesemigroup}
S_{h}(t):=e^{-t\, A_{h}}=\dfrac{1}{2 \pi i}\int_{\mathcal{C}'} e^{t\,\lambda}(\lambda I +A_{h})^{-1} d \lambda\;
,\qquad  \qquad \qquad  t>0
\end{eqnarray}
where $\mathcal{C}'$  denotes a path that surrounds the spectrum of $-A_{h}$.
As for the continuous case, \thmref{th1e}  and Duhamel's principle  ensure   the existence and uniqueness of the  solution of (\ref{adrh})  represented by the following  integral equations (mild form)
\begin{eqnarray}
\label{adrd}
X_{h}(t)=S_{h}(t)X_{0 h}+ \int_{0}^{t}S_{h}(t-s)P_{h}F(X_{h}(s))ds,\qquad \qquad \ t \in \left[0,T \right].
\end{eqnarray}
The solution $X$ converges to the semi-discrete solution $X_h$ according to the following theorem.
\begin{theorem}
\label{Ft}
Let  $\mathcal{B}\subset V$ be bounded, consider the solution $X$ of \eqref{adr6} and the semi-discrete  solution $X_h$  of \eqref{adrh} represented by \eqref{fvfd} or \eqref{adrd} 
in the  interval $[0, t^{*}]$, $ t^{*}= t^{*}(\mathcal{B})$ defined in \thmref{th1e},
such that  $X(t) \in \mathcal{B}$ and $ X_h(t)\in \mathcal{B}\bigcap V_h $   for  all $t \leq T \leq  t^{*}$.
 We assume that the unique mild solution $X$  of \eqref{adr6}  is the classical solution (i.e. $X$ is twice continuously differentiable with respect 
to $\mathbf{x}$ and differentiable with respect to $t$), \assref{meshassumption} is satisfied  and the reaction 
function $F$ satisfies  \eqref{lip}. 
Furthermore assume that $X_{0} \in C(\overline{\Omega}) \bigcap \mathcal{B}$, $X_{0h} \in V_h\bigcap \mathcal{B}$  and $f(\mathbf{x},u)$ is  differentiable respect to $\mathbf{x}$ and $u$  with
\begin{eqnarray}
\label{Extraf}
 \vert f_\mathbf{x}(\mathbf{x},u)\vert +\vert f_u(\mathbf{x},u)\vert \leq C(1+\vert u\vert^{\gamma}),\; \quad \forall \; \mathbf{x} \in \Omega, \;\; u \in \mathbb{R},
\end{eqnarray}
for function of type $f$ of type \eqref{lipf} and 
\begin{eqnarray}
\label{extraff}
 \vert f_\mathbf{x}(\mathbf{x},u)\vert +\vert f_u(\mathbf{x},u)\vert \leq C(1+ \vert u\vert+ \vert u\vert^{\gamma}),\; \quad \forall \; \mathbf{x} \in \Omega, \;\; u \in \mathbb{R},
\end{eqnarray}
for function of type $f$ of type \eqref{lipff},     
then the following estimate holds
\begin{eqnarray*}
 \Vert X(t)-X_{h}(t) \Vert_{0,h} \leq  C(\mathcal{B}) (\Vert X_0- X_{0h} \Vert_{0,h} +h),\;\;\;\;\;\;\;\; \forall t \in [0, T],
\end{eqnarray*}
where $C=C( \mathcal{B},\Omega,X,F,\mathbf{D},\mathbf{q}, T,\zeta_{1}).$
\end{theorem}
  Before given the proof, let us give this lemma which will be very useful in the proof.
  
As for  the elliptic case \cite[ Proof of Theorem 2.3 or Theorem 3.8]{FV} or linear parabolic case \cite[Proof of Theorem 4.1
]{FV}, we have the following fluxes consistency.
\begin{lemma}
\label{consistency}
[\textbf{ Fluxes consistency}]\\
Let $R_{i,\sigma}(t)$ and $r_{i,\sigma}(t)$ be respectively the errors of diffusion flux and advective flux  through the the edge $\sigma$ (interface of control volume $A_i$  
and control volume $A_j$, or edge of control volume $A_i$  if $\sigma \subset \partial \Omega $) at time $t$ given by 
 \begin{eqnarray}
R_{i,\sigma}(t)=\dfrac{1}{\mathrm{mes}(\sigma)}\left[ \dfrac{\mathrm{mes}(\sigma) \mu_{\sigma}}{d(i,j)}\left(X(\mathbf{x}_i,t)-X(\mathbf{x}_j,t) \right)-\int_{\sigma} -\mathbf{D}\nabla X \cdot \mathbf{n}_{i,\sigma} d\sigma\right] ,\\
\newline\\ \nonumber
r_{i,\sigma}(t)= \dfrac{1}{\mathrm{mes}(\sigma)}\left[ q_{i,\sigma}X( \mathbf{x}_{\sigma,+},t)-\int_{\sigma} \mathbf{q} X(t)\cdot  \mathbf{n}_{i,\sigma}\right],
\end{eqnarray}
 \footnote{Remember that $X (\mathbf{x}_j,t)= X(\mathbf{x}_\sigma,t)=0$ if $\sigma \subset \partial \Omega $.} where 
\begin{eqnarray}
\left\lbrace \begin{array}{l}
\mathbf{x}_{\sigma,+}= \left\lbrace \begin{array}{l} \mathbf{x}_{i}\;\; \text{if}\; \mathbf{q} \cdot  \mathbf{n}_{\sigma}\geq 0,\\
\newline\\
\mathbf{x}_{j} \;\;\text{if}\;\;\;\mathbf{q} \cdot  \mathbf{n}_{\sigma}\, <0 ,
\end{array}\right.
\sigma= i\vert j,\\
\newline\\
\mathbf{x}_{\sigma,+}= \left\lbrace \begin{array}{l} \mathbf{x}_{i}\;\; \text{if}\;\quad \mathbf{q} \cdot  \mathbf{n}_{\sigma}\geq 0,\\
\newline\\
\mathbf{x}_{\sigma} ,\;\; \quad \mathbf{x}_{\sigma}\in \partial \Omega \;\;\text{if}\;\;\;\mathbf{q} \cdot \mathbf{n}_{\sigma} <0
\end{array}\right.
\sigma \in \mathcal{E}_{i}\cap \partial \Omega.
\end{array}\right.
\end{eqnarray}
Under the \assref{meshassumption}, if $X$ is twice continuously differentiable with respect 
to $\mathbf{x}$,  there exists three positive constants $C_{2}$,  $C_{2}'$ and $C_{3}$ such that
\begin{eqnarray}
\label{conti}
\left\lbrace \begin{array}{l}
\vert R_{i,\sigma}(t)\vert \leq C_{2}\,(\mathbf{D},X,T,\Omega)\, h,\\
\newline\\
\vert r_{i,\sigma}(t)\vert \leq C_{2}'\,(\mathbf{q},X,T,\Omega) \,h,\\
\newline\\
\vert R_{i,\sigma}(t)\vert+\vert r_{i,\sigma}(t)\vert \leq C_{3}\,( \mathbf{q},\mathbf{D},X,T \Omega)\, h.
\end{array}\right.
\end{eqnarray}
\end{lemma}
\begin{proof}
 The proof of this lemma can be done in the same manner as the one in \cite{FV} in 1 D. Let us provide some details. By setting 
 \begin{eqnarray}
  \overline{F}_{i,\sigma}=\int_{\sigma} -\mathbf{D}\nabla X \cdot \mathbf{n}_{i,\sigma} d\sigma, \,\,\,F^{*}_{i,\sigma}= \dfrac{\mathrm{mes}(\sigma) \mu_{\sigma}}{d(i,j)}\left(X(\mathbf{x}_i,t)-X(\mathbf{x}_j,t) \right)=-\tau_\sigma(X(\mathbf{x}_j,t)-X(\mathbf{x}_i,t)),
 \end{eqnarray}
let us prove that there exists $C_2 = C_2( \mathbf{D}, X,T,\Omega)$  such that
 \begin{eqnarray}
  F^{*}_{i,\sigma}=\overline{F}_{i,\sigma}+R_{i,\sigma}(t), \,\,\, \text{where}\,\,\,\, \vert R_{i,\sigma} \vert \leq C_2 \mathrm{mes}(\sigma) h.
 \end{eqnarray}
 \textbf{Case 1}: $\sigma = A_i \mid A_j \in  \mathcal{E}_{i}$.  Let  $\mathbf{y}_\sigma= \sigma\bigcap \mathcal{D}_{i,\sigma}=\sigma\bigcap \mathcal{D}_{j,\sigma}$, where
 $\mathcal{D}_{i,\sigma}$  and $\mathcal{D}_{j,\sigma}$ are respectively
 the straight lines perpendicular to $\sigma = A_i \mid A_j $ with respect to
the inner product induced by $\mathbf{D}_{A_i}^{-1}$ and $\mathbf{D}_{A_j}^{-1}$.
 Let us  set   
\begin{eqnarray}
F^{*,i}_{i,\sigma}=-D_{i,\sigma} \mathrm{mes}(\sigma) \dfrac{X(\mathbf{y}_\sigma,t)-X(\mathbf{x}_i,t)}{d_{i,\sigma}},\,\,\,
   F^{*,j}_{j,\sigma}=-D_{j,\sigma} \mathrm{mes}(\sigma) \dfrac{X(\mathbf{y}_\sigma,t)-X(\mathbf{x}_j,t)}{d_{i,\sigma}}.
 \end{eqnarray}
 Since  $X$ is twice continuously differentiable with respect 
to $\mathbf{x}$, using \assref{meshassumption} and \defref{admissiblemesh} and combined with Taylor expansion yields
 \begin{eqnarray}
 \label{feq}
  F^{*,i}_{i,\sigma}=\overline{F}_{i,\sigma}+ t_{i,\sigma}^{i}\,\,\,\,\, \text{with}\,\,\,\,\vert t_{i,\sigma}^{i}\vert \leq \alpha_1 \mathrm{mes}(\sigma) h\\
  F^{*,j}_{j,\sigma}=\overline{F}_{j,\sigma}+ t_{j,\sigma}^{j}\,\,\,\,\, \text{with}\,\,\,\,\vert t_{j,\sigma}^{j}\vert \leq \alpha_1 \mathrm{mes}(\sigma) h
 \end{eqnarray}
 First of all let us proof  the \eqref{feq}. As the restriction of  $\mathbf{D}$ to $A_i$ is differentiable, there exists a positive constant $C=C(\mathbf{D})$ such that
\begin{eqnarray}
 \mathbf{D}(\mathbf{x})= \mathbf{D}_{A_i}+\mathbf {m}_i\;\;\;\;\;\; \forall \,\mathbf{x} \in \overline{A_i},\,\,\,\, \text{with}\;\;\;\; \Vert \mathbf{m}_i \Vert_{\mathbb{R}^{d \times d}} \leq C h,
\end{eqnarray}
So,  using the fact that $X$ is dfferentiable yields
\begin{eqnarray}
\label{meand}
 \int_{\sigma} \mathbf{D}\nabla X \cdot \mathbf{n}_{i,\sigma} d\sigma = \int_{\sigma} \mathbf{D}_{A_i}\nabla X \cdot \mathbf{n}_{i,\sigma} d\sigma + M_i,\,\, \text{with}\;\;\;\; \vert M_i \vert \leq C(\Omega,T,\mathbf{D}) \mathrm{mes}(\sigma) h.
\end{eqnarray} 
As $X$ is twice continuously differentiable with respect to $\mathbf{x}$, Taylor expansion yields
\begin{eqnarray*}
 X(\mathbf{y}_\sigma,t)-X(\mathbf{x}_i,t)=\nabla X \cdot (\mathbf{y}_\sigma-\mathbf{x}_i)+\int_{0}^{1}H(X)(t \mathbf{x}_i+(1-t)\mathbf{y}_\sigma) (\mathbf{y}_\sigma-\mathbf{x}_i) \cdot(\mathbf{y}_\sigma-\mathbf{x}_i) t dt.
\end{eqnarray*}
where $H(X)(z)$ denotes the Hessian matrix of $X$ at  point $z$. Note that
\begin{eqnarray}
 \mathbf{y}_\sigma-\mathbf{x}_i=d_{i,\sigma} \mathbf{n}_{i,\sigma}^{*},
\end{eqnarray}
 where $\mathbf{n}_{i,\sigma}^{*}$ is  the  normal unit vector to $\sigma$ outward to $A_i$  with respect to
the inner product induced by $\mathbf{D}_{A_i}^{-1}$, which is different to $\mathbf{n}_{i,\sigma}$(the  normal unit vector to $\sigma$ outward to $A_i$ with  the  with respect to
the inner product of  $\mathbb{R}^{d}$). We therefore have
\begin{eqnarray}
\label{conf}
 \dfrac {X(\mathbf{y}_\sigma,t)-X(\mathbf{x}_i,t)}{d_{i,\sigma}}=\nabla X \cdot \mathbf{n}_{i,\sigma}^{*}+\frac{1}{d_{i,\sigma}}\int_{0}^{1}H(X)(t \mathbf{x}_i+(1-t)\mathbf{y}_\sigma) (\mathbf{y}_\sigma-\mathbf{x}_i) \cdot(\mathbf{y}_\sigma-\mathbf{x}_i) t dt.
\end{eqnarray}
 Note that by definition of the scalar product induced by $\mathbf{D}_{A_i}^{-1}$ (\defref{admissiblemesh}), we have
 \begin{eqnarray}
 (\mathbf{y}_\sigma-\mathbf{a})^{T} \mathbf{D}_{A_i}^{-1} \mathbf{n}_{i,\sigma}^{*}=0,\;\;\;\; \forall \mathbf{a} \in \sigma.
 \end{eqnarray}
Since
 \begin{eqnarray}
 (\mathbf{y}_\sigma-\mathbf{a})^{T} \mathbf{D}_{A_i}^{-1} (\mathbf{D}_{A_i}\mathbf{n}_{i,\sigma}) =  (\mathbf{y}_\sigma-\mathbf{a})^{T}\mathbf{n}_{i,\sigma}=0,\;\;\;\; \forall \mathbf{a} \in \sigma,
 \end{eqnarray}
 we can therefore take
 \begin{eqnarray}
  \mathbf{n}_{i,\sigma}^{*}= \dfrac{\mathbf{D}_{A_i}\mathbf{n}_{i,\sigma}}{\vert \mathbf{D}_{A_i}\mathbf{n}_{i,\sigma}\vert}.
 \end{eqnarray}
So, \eqref{conf} becomes
\begin{eqnarray}
\label{conf1}
\lefteqn{\vert \mathbf{D}_{A_i}\mathbf{n}_{i,\sigma}\vert \dfrac {X(\mathbf{y}_\sigma,t)-X(\mathbf{x}_i,t)}{d_{i,\sigma}}}\\
&=&\nabla X \cdot \mathbf{D}_{A_i}\mathbf{n}_{i,\sigma}+\dfrac{\vert \mathbf{D}_{A_i}\mathbf{n}_{i,\sigma}\vert}{d_{i,\sigma}}\int_{0}^{1}H(X)(t \mathbf{x}_i+(1-t)\mathbf{y}_\sigma) (\mathbf{y}_\sigma-\mathbf{x}_i) \cdot(\mathbf{y}_\sigma-\mathbf{x}_i) t dt.
\end{eqnarray}
Using the fact  that $\mathbf{D}_{A_i}$ is symmetric, we have 
 \begin{eqnarray}
 \label{ega}
  \mathbf{D}_{A_i} \nabla X \cdot \mathbf{n}_{i,\sigma}= (\mathbf{D}_{A_i} \nabla X )^{T} \mathbf{n}_{i,\sigma}= (\nabla X )^{T} (\mathbf{D}_{A_i})^{T}  \mathbf{n}_{i,\sigma}=  (\nabla X )^{T} \mathbf{D}_{A_i} \mathbf{n}_{i,\sigma}= \nabla X \cdot \mathbf{D}_{A_i} \mathbf{n}_{i,\sigma}.
 \end{eqnarray}
 By integrating \eqref{conf1} over $\sigma$,  using the fact  that $X$ is twice differentiable respect to $\mathbf{x}$, combining  \eqref{meand} and \eqref{ega} yields
 \begin{eqnarray}
 \label{feqn}
  F^{*,i}_{i,\sigma}=\overline{F}_{i,\sigma}+ t_{i,\sigma}^{i}\,\,\,\,\, \text{with}\,\,\,\,\vert t_{i,\sigma}^{i}\vert \leq \alpha_1 (\mathbf{D},X,T,\Omega) \mathrm{mes}(\sigma) h,
 \end{eqnarray}
 and  this conclude the proof of \eqref{feq}.
 
Let us continue with the proof of  the lemma.  The continuity of the fluxes, i.e. $\overline{F}_{i,\sigma}=-\overline{F}_{j,\sigma}$ implies
\begin{eqnarray}
 X(\mathbf{y}_\sigma,t)= -\dfrac{t_{i,\sigma}^{i}+t_{j,\sigma}^{j}}{\left(\frac{D_{i,\sigma}}{d_{i,\sigma}}+\frac{D_{j,\sigma}}{d_{j,\sigma}}\right)\mathrm{mes}(\sigma)}+\dfrac{1}{\left(\frac{D_{i,\sigma}}{d_{i,\sigma}}+\frac{D_{j,\sigma}}{d_{j,\sigma}}\right)}\left( 
 \frac{D_{i,\sigma}}{d_{i,\sigma}} X(\mathbf{x}_i,t)+\frac{D_{j,\sigma}}{d_{j,\sigma}} X(\mathbf{x}_j,t)\right).
\end{eqnarray}
We therefore have
\begin{eqnarray}
\label{feq1}
 F^{*,i}_{i,\sigma}=  \frac{D_{i,\sigma}}{d_{i,\sigma}}\dfrac{t_{i,\sigma}^{i}+t_{j,\sigma}^{j}}{\left(\frac{D_{i,\sigma}}{d_{i,\sigma}}+\frac{D_{j,\sigma}}{d_{j,\sigma}}\right)}-\tau_\sigma(X(\mathbf{x}_j,t)-X(\mathbf{x}_i,t)).
\end{eqnarray}
Combining \eqref{feq} and\eqref{feq1} yields 
\begin{eqnarray}
 \overline{F}_{i,\sigma}&=&F^{*,i}_{i,\sigma}-t_{i,\sigma}^{i}=   \frac{D_{i,\sigma}}{d_{i,\sigma}}\dfrac{t_{i,\sigma}^{i}+t_{j,\sigma}^{j}}{\left(\frac{D_{i,\sigma}}{d_{i,\sigma}}+\frac{D_{j,\sigma}}{d_{j,\sigma}}\right)}+F^{*}_{i,\sigma}-t_{i,\sigma}^{i}\\
    &=&F^{*}_{i,\sigma}-R_{i,\sigma}
 \end{eqnarray}
 where 
 \begin{eqnarray}
  -R_{i,\sigma}&=&\frac{D_{i,\sigma}}{d_{i,\sigma}}\dfrac{t_{i,\sigma}^{i}+t_{j,\sigma}^{j}}{\left(\frac{D_{i,\sigma}}{d_{i,\sigma}}+\frac{D_{j,\sigma}}{d_{j,\sigma}}\right)}-t_{i,\sigma}^{i}\\
 \vert R_{i,\sigma}\vert &\leq&  \vert t_{i,\sigma}^{i}\vert +\vert \frac{D_{i,\sigma}}{d_{i,\sigma}}\dfrac{t_{i,\sigma}^{i}+t_{j,\sigma}^{j}}{\left(\frac{D_{i,\sigma}}{d_{i,\sigma}}+\frac{D_{j,\sigma}}{d_{j,\sigma}}\right)}\vert\\
    &\leq& \vert t_{i,\sigma}^{i}\vert+ \dfrac{D_{i,\sigma}d_{j,\sigma}}{d_{j,\sigma} D_{i,\sigma}+d_{i,\sigma} D_{j,\sigma}} \left(\vert t_{i,\sigma}^{i}\vert +\vert t_{j,\sigma}^{j}\vert\right)     \\  
 &\leq& 3 \alpha_1 \mathrm{mes} (\sigma) h,
 \end{eqnarray}
since
\begin{eqnarray}
 \dfrac{D_{i,\sigma}d_{j,\sigma}}{d_{j,\sigma} D_{i,\sigma}+d_{i,\sigma} D_{j,\sigma}} \leq 1.
\end{eqnarray}
\textbf{Case 2}: $\sigma \in  \partial \Omega \bigcap \mathcal{E}_i$.  As we have in \textbf{Case 1}, since  $X$ is twice continuously differentiable with respect 
to $\mathbf{x}$, using \assref{meshassumption} and \defref{admissiblemesh} combined with Taylor expansion yields
\begin{eqnarray}
  F^{*}_{i,\sigma}=\dfrac{\mathrm{mes}(\sigma) D_{i,\sigma}}{d_{i,\sigma}}\left(X(\mathbf{x}_i,t) \right)=\tau_\sigma X(\mathbf{x}_i,t)= \overline{F}_{i,\sigma}+ R_{i,\sigma},\,\,\text{with}\,\, 
  \vert R_{i,\sigma}\vert \leq  \alpha \mathrm{mes}(\sigma) h.
 \end{eqnarray}
 To conclude the proof of the diffusion error $R_{i,\sigma}(t)$, we can take $C_2= 3 \alpha_1$.
 
 The  proof of the convection error $r_{i,\sigma}(t)$  is done  by using Tayor expansion of $X$ and $\mathbf{q}$ (since $\mathbf{q}$ is assumed to be differentiable according to \assref{meshassumption}).  
 Some details can be found in \cite{FV}.\\
 \end{proof}

Let us now prove our first main result in \thmref{Ft}.

\begin{proof}
 Integrating the shifted version (by adding $c_{0}X$ in both size) of equation (\ref{adr6}) over each control volume $ A_i \in \mathcal{T}$  
and using the divergence theorem yields
\begin{eqnarray}
\label{vdadr}
 \int_{A_i}X_{t}(\mathbf{x},t)d\mathbf{x} -\underset{\sigma \in \mathcal{E}_{i}}{\sum}\int_{\sigma}\left( \mathbf{D}\nabla X -\mathbf{q} X\right)\cdot \mathbf{n}_{\sigma} d\sigma + c_{0}\int_{A_i}X d \mathbf{x}= \int_{A_i}f(\mathbf{x},X(\mathbf{x},t))d\mathbf{x}.
\end{eqnarray}
For $t\in \left[0,T \right],\;\;A_i \in \mathcal{T}$  and $\sigma \in \mathcal{E}_{i}$ using the same notation as in \cite{FV}, let us set
\begin{eqnarray}
\label{mean}
\left\lbrace \begin{array}{l}
p_{i}(t)=X(\mathbf{x}_{i},t)-\dfrac{1}{\mathrm{mes}(A_i)}\int_{A_i}X(\mathbf{x},t)d\mathbf{x},\\
\newline\\
\varrho_{i}(t)= \dfrac{1}{\mathrm{mes}(A_i)}\int_{A_i} f(\mathbf{x},X(\mathbf{x},t))d\mathbf{x}-f(\mathbf{x}_{i},X_{i}(t)).
\end{array}\right.
\end{eqnarray}
As we have assumed that the unique solution $X$ of (\ref{adr6}) is the regular, Taylor expansion yields 
\begin{eqnarray}
\label{deri}
\left\lbrace \begin{array}{l}
 X_{t}(\mathbf{x},t)=X_{t}(\mathbf{x}_{i},t)+s_{i}(\mathbf{x},t),\quad \quad \quad \vert s_{i}(\mathbf{x},t)\vert \leq C_{1}\,(X,T)\, h\\
\newline\\
\int_{A_i}X_{t}(\mathbf{x},t)d\mathbf{x}=\mathrm{mes}(A_i)X_{t}(\mathbf{x}_{i},t)+S_{i},\quad  S_{i}=\int_{A_i}s_{i}(\mathbf{x},t)d\mathbf{x}, \quad \vert S_{i}\vert \leq \mathrm{mes}(A_i)C_{1}\,(X,T)\,h.
\end{array}\right.
\end{eqnarray}
 Using again the regularity of the solution $X$, we also have 
 $$ X(\mathbf{x},t)=X(\mathbf{x}_{i},t)+s'_{i}(\mathbf{x},t),\quad \quad \quad \vert s'_{i}(\mathbf{x},t)\vert \leq C_{1}'\,(X,T)\, h,\\ $$
 therefore
 \begin{eqnarray}
\label{conti}
\vert p_{i}(t) \vert \leq C_{3}'\,(X,T)\,h.
\end{eqnarray}

Using  the expressions \eqref{mean} and \eqref{deri} in \eqref{vdadr} yields the following decomposition  of our initial continuous  problem \eqref{vdadr}
\begin{eqnarray}
\label{vdadr1}
&&\mathrm{mes}(A_i)X_{t}(\mathbf{x}_{i},t)+S_{i} -\underset{\sigma \in \mathcal{E}_{i}}{\sum}\mathrm{mes}(\sigma) \left(R_{i,\sigma}(t)+ r_{i,\sigma}(t)\right)+\underset{\sigma \in \mathcal{E}_{i}}{\sum}\left[ \dfrac{\mathrm{mes}(\sigma) \mu_{\sigma}}{d(i,j)}\left(X(\mathbf{x}_i,t)-X(\mathbf{x}_j,t) \right)\right]
 \nonumber\\
 &&+\underset{\sigma \in \mathcal{E}_{i}}{\sum}\left[ q_{i,\sigma}X( \mathbf{x}_{\sigma,+},t)\right]
  + c_{0} \mathrm{mes}(A_i) \left(X(\mathbf{x}_i,t)-p_i(t)\right)= \int_{A_i}f(\mathbf{x},X(\mathbf{x},t))d\mathbf{x}.
\end{eqnarray}

Let $X_h(t) \in V_h$ solution of \eqref{varia1} \footnote{ also solution of \eqref{adrh} or \eqref{fvfd}} such that $X_h(t)(\mathbf{x}_i)=X_i(t)$.
Subtracting the first equation of (\ref{fvfd}) from (\ref{vdadr1}) yields
\begin{eqnarray}
\label{derr1}
\left\lbrace \begin{array}{l}
 \mathrm{mes}(A_i)\dfrac{d e_{i}(t)}{dt}+\underset{\sigma \in \mathcal{E}_{i}}{\sum}G_{i,\sigma}(t)+W_{i,\sigma}(t)+c_{0}\mathrm{mes}(A_i) e_{i}(t)
\newline\\
 \quad  \quad \quad  \quad   \quad  \quad  \quad  \quad   \quad  \quad \quad  =\int_{A_i}\left( f(\mathbf{x},X(\mathbf{x},t))-f(\mathbf{x}_{i},X_{i}(t))\right)d\mathbf{x}\\
\newline\\
\quad  \quad \quad  \quad   \quad  \quad  \quad  \quad   \quad  \quad \quad  +c_{0}\mathrm{mes}(A_i)p_{i}(t)+\underset{\sigma \in \mathcal{E}_{i}}{\sum}\mathrm{mes}(\sigma)(R_{i,\sigma}(t)+r_{i,\sigma}(t))-S_{i}(t),\quad \forall A_i\in \mathcal{T}
\end{array}\right.
\end{eqnarray}
with 
\begin{eqnarray}
\label{term}
\left\lbrace \begin{array}{l}
e_{i}(t)=X(\mathbf{x}_{i},t)-X_{i}(t)=X(\mathbf{x}_{i},t)-X_h(t)(\mathbf{x}_i),\quad \quad \quad t\in \left[ 0,T\right], \\
\newline\\
G_{i,\sigma}(t)=-\tau_{\sigma}(e_{j}(t)-e_{i}(t)),\quad \quad \sigma= i \vert j, \\
\newline\\
G_{i,\sigma}(t)=\tau_{\sigma}e_{i}(t),\quad  \sigma \in\mathcal{E}_{i}\cap \partial \Omega,\\
\newline\\
W_{i,\sigma}(t)=q_{i,\sigma}(X(\mathbf{x}_{\sigma,+},t)-X_{\sigma,+}(t)).
\end{array}\right.
\end{eqnarray}
Multipling equation (\ref{derr1}) by $e_{i}(t)$ and  summing for $A_i \in \mathcal{T}$ yields
\begin{eqnarray}
\label{derr2}
\left\lbrace \begin{array}{l}
 \underset{A_i \in \mathcal{T}}{\sum}\left[  \dfrac{\mathrm{mes}(A_i)}{2}\dfrac{d(e_{i}^{2}(t))}{dt}+\underset{\sigma \in \mathcal{E}_{i}}{\sum}e_{i}(t)(G_{i,\sigma}(t)+W_{i,\sigma}(t))+c_{0}\,\mathrm{mes}(A_i) e_{i}^{2}(t)\right] \\
\newline\\
  =\underset{A_i \in \mathcal{T}}{\sum}e_{i}(t)\left[ \int_{A_i}\left(f(\mathbf{x},X(\mathbf{x},t))-f(\mathbf{x}_{i},X_{i}(t))\right)d\mathbf{x} \right] \\
\newline\\
  +\underset{A_i \in \mathcal{T}}{\sum}\left[c_{0} \,\mathrm{mes}(A_i)p_{i}(t)e_{i}(t)+\underset{\sigma \in \mathcal{E}_{i}}{\sum}\mathrm{mes}(\sigma)e_{i}(t)(R_{i,\sigma}(t)+r_{i,\sigma}(t))-e_{i}(t)S_{i}(t)\right].
\end{array}\right.
\end{eqnarray}
Using the fact that $f$ is differentiable with respect to $X$ and $\mathbf{x}$, Taylor expansion yields
\begin{eqnarray}
\label{Rlip}
\lefteqn{ \mathrm{mes}(A_i)\varrho_{i}(t)} \nonumber\\
&=& \int_{A_i} f(\mathbf{x},X(\mathbf{x},t))d\mathbf{x}-\mathrm{mes}(A_i) f(\mathbf{x}_{i},X_i(t))\nonumber\\
   & =& \int_{A_i}\left( f(\mathbf{x},X(\mathbf{x},t))-f(\mathbf{x}_{i},X_i(t))\right)d\mathbf{x}, \nonumber\\
& =& \mathrm{mes}(A_i)\left(f(\mathbf{x}_{i},X(\mathbf{x}_{i},t))- f(\mathbf{x}_{i},X_i(t))\right)+ \int_{A_i} Z_{1}(\mathbf{x},t)(X(\mathbf{x},t)-X(\mathbf{x}_{i},t))d\mathbf{x} \nonumber\\
&& +\int_{A_i}Z_{2}(\mathbf{x},t)(\mathbf{x}-\mathbf{x}_{i})d\mathbf{x} \nonumber\\
&=& \mathrm{mes}(A_i)\left( f(\mathbf{x}_{i},X(\mathbf{x}_{i},t))- f(\mathbf{x}_{i},X_i(t))\right)+\kappa (\mathbf{x}_{i},X,f),
\end{eqnarray}
where 
\begin{eqnarray*}
\kappa (\mathbf{x}_{i},X,f)= \int_{A_i} Z_{1}(\mathbf{x},t)(X(\mathbf{x},t)-X(\mathbf{x}_{i},t))d\mathbf{x}+\int_{A_i}Z_{2}(\mathbf{x},t)(\mathbf{x}-\mathbf{x}_{i})d\mathbf{x}\\
\newline\\
 Z_{1}(\mathbf{x},t)= \int_{0}^{1}\dfrac{\partial f}{\partial X}\left(\mathbf{x}_{i}+\tau(\mathbf{x}-\mathbf{x}_{i}),X(\mathbf{x}_{i},t)+\tau (X(\mathbf{x},t)-X(\mathbf{x}_{i},t))\right) d\tau\\
\newline\\
Z_{2}(\mathbf{x},t)= \int_{0}^{1}\dfrac{\partial f}{\partial \mathbf{x}}\left(\mathbf{x}_{i}+\tau(\mathbf{x}-\mathbf{x}_{i}),X(\mathbf{x}_{i},t)+\tau (X(\mathbf{x},t)-X(\mathbf{x}_{i},t))\right) d\tau.\\
\end{eqnarray*}
As we have assumed that the solution $X(t) \in \mathcal{B}$ is differentiable with respect  $\mathbf{x}$ and  $f$ differentiable  with respect to the two variables  with derivatives satisfying 
\eqref{Extraf} or \eqref{extraff}, one more Taylor expansion  yields 
$$
\vert \kappa (\mathbf{x}_{i},t,X,f)\vert \leq  \mathrm{mes}(A_i) C_{4}(\mathcal{B},\Omega,f,T,X)\,h.
$$
Using \eqref{lipf}, \eqref{lipff} or \propref{alip}, and the fact that $X(t) \in \mathcal{B}$ and $X_h(t) \in \mathcal{B}$ allow to have
\begin{eqnarray}
\label{react}
  \mathrm{mes}(A_i)\varrho_{i}(t)\leq \mathrm{mes}(A_i)\left( C_{4}'( \mathcal{B},\Omega,f,T,X)\vert X(\mathbf{x}_{i},t)-X_{i}(t)\vert+ C_{4}( \mathcal{B}, \Omega,T,X)\,h \right). 
\end{eqnarray}
Let $e_{h}(t) \in V_h$  a piecewise constant function defined by
\begin{eqnarray}
 e_{h}(t)(\mathbf{x}_{i})=e_{i}(t)= X(\mathbf{x}_i,t)-X_h(t)(\mathbf{x}_i) \quad A_i\in\mathcal{T},\quad  t\in \left[0,T \right].
\end{eqnarray}

Since  

\begin{eqnarray}
 \tau_\sigma=\dfrac{\mathrm{mes}}{d_\sigma} \mu_\sigma.
\end{eqnarray}
Using \eqref{term}, reordering the summation and  the fact that the  transmissibility is symmetric, i.e.  $\tau_{i\vert j}=\tau_{j\vert i}$, we have
\begin{eqnarray}
\Vert e_{h}(t)\Vert_{1,h}^{2}&:=&\underset{A_i \in \mathcal{T}}{\sum}\underset{\sigma \in \mathcal{E}_{i}}{\sum} e_{i}(t)G_{i,\sigma}(t) \nonumber,\\
                       &=& \underset{\sigma \in \mathcal{E}}{\sum} \vert D_{\sigma}e_{h}(t)\vert^{2}\dfrac{\mathrm{mes}(\sigma) \mu_\sigma}{d_{\sigma}}.  
\end{eqnarray}
Note  that  $\mu_{\sigma}$ is defined in \eqref{transn11} and \eqref{transn1}.
As in the proof of Theorem \ref{coert}, using the regularity of  the mesh $\mathcal{T}$ ($\zeta_1 h \leq d_{i,\sigma} \leq h$), we have
\begin{eqnarray}
 C_{5}\leq\mu_{\sigma}\,\leq \,C_{5}'\,,\quad \quad \forall \sigma \in \mathcal{E}.
\end{eqnarray}
 
So
\begin{eqnarray}
\label{flux1}
C_{5}\Vert e_{h}(t)\Vert_{1,\mathcal{T}}^{2}\leq \Vert e_{h}(t)\Vert_{1,h}^{2}\leq C_{5}'\,\Vert e_{h}(t)\Vert_{1,\mathcal{T}}^{2}.
\end{eqnarray}
Note that
\begin{eqnarray}
\left\lbrace \begin{array}{l}
\vert D_{\sigma}e_{h}(t)\vert =\vert e_{i}(t)-e_{j}(t)\vert,\quad \quad \text{if}\quad \quad  \sigma= i\vert j,\\
\newline\\
\vert D_{\sigma}e_{h}(t)\vert =\vert e_{i}(t)\vert,\quad \quad \text{if}\quad \quad  \sigma \in \mathcal{E}_{i}\cap \partial \Omega.\\
\end{array}\right.
\end{eqnarray}
 Setting $e_{\sigma,+}(t)=X(\mathbf{x}_{\sigma,+},t)-X_{\sigma,+}(t),
$
as in the proof of Theorem \ref{coert},  using the same technique as in   bilinear form $b_h^2(. )$ (see \eqref{egaly1}-\eqref{eq2}) yields 
\begin{eqnarray}
\label{flux2}
 \underset{A_i \in \mathcal{T}}{\sum}\underset{\sigma \in \mathcal{E}_{i}}{\sum}e_{i}(t)W_{i,\sigma}(t)&=& \underset{i \in \mathcal{T}}{\sum}\underset{\sigma \in \mathcal{E}_{i}}{\sum} q_{i,\sigma}e_{i}(t)(X(\mathbf{x}_{\sigma,+},t)-X_{\sigma,+}(t)) \nonumber\\
               &=&\underset{A_i \in \mathcal{T}}{\sum}\underset{\sigma \in \mathcal{E}_{i}}{\sum} q_{i,\sigma}e_{i}(t)e_{\sigma,+}(t) 
\geq 0.
\end{eqnarray}
Using (\ref{flux2}) and \eqref{react} in the expression (\ref{derr2}) yields
\begin{eqnarray}
\label{derr3}
\left\lbrace \begin{array}{l}
\dfrac{1}{2}\underset{A_i \in \mathcal{T}}{\sum}  \mathrm{mes}(A_i)\dfrac{d(e_{i}^{2}(t))}{dt}+\Vert e_{h}(t)\Vert_{1,h}^{2}+c_{0}\Vert e_{h}(t)\Vert_{0,h}^{2}\leqslant C_{4}'(\mathcal{B})\Vert e_{h}(t)\Vert_{0,h}^{2} \\
\newline\\
+C_{4}(\mathcal{B})\,h\,\underset{A_i \in \mathcal{T}}{\sum}\mathrm{mes}(A_i) \vert e_{i}(t)\vert

 + c_{0} C_{3}'\,h\,\underset{A_i\in \mathcal{T}}{\sum} \mathrm{mes}(A_i)\vert e_{i}(t)\vert\\
 \newline\\
 +\underset{A_i\in \mathcal{T}}{\sum}\underset{\sigma \in \mathcal{E}_{i}}{\sum}\mathrm{mes}(\sigma)e_{i}(t)(R_{i,\sigma}(t)+r_{i,\sigma}(t))+ C_{1}\,h\,\underset{A_i\in \mathcal{T}}{\sum}\mathrm{mes}(A_i)\vert e_{i}(t)\vert.
\end{array}\right.
\end{eqnarray}
The   continuity of the diffusion and advection flux at each interface  yields
$$R_{i,\sigma}(t)=-R_{j,\sigma}(t),\;\;\;\quad\;r_{i,\sigma}(t)=-r_{j,\sigma}(t),\quad \quad \text{for}\; \sigma= i\vert j \in \mathcal{E}_{int}.$$
Set $$R_{\sigma}(t)=\vert R_{i,\sigma}(t)\vert,\; \; r_{\sigma}(t)=\vert R_{i,\sigma}(t)\vert,\;\quad \quad A_i \in \mathcal{T},\quad \quad \sigma
\in \mathcal{E}_{int}.$$
Using  the Cauchy-Schwarz inequality  as in \cite{FV} for stationary elliptic problems, and
  reordering the summation over the edges  yields
\begin{eqnarray*}
\label{derr44}
\lefteqn{\underset{A_i\in \mathcal{T}}{\sum}\underset{\sigma \in \mathcal{E}_{i}}{\sum}\mathrm{mes}(\sigma)e_{i}(t)(R_{i,\sigma}(t)+r_{i,\sigma}(t))}&&\\
&\leqslant& \underset{\sigma \in \mathcal{E}}{\sum}
\mathrm{mes}(\sigma)D_{\sigma} e_{h}(t)(R_{\sigma}(t)+r_{\sigma}(t))\\
&\leqslant& \left( 
\underset{\sigma \in \mathcal{E}}{\sum}\dfrac{\mathrm{mes}(\sigma)}{d_{\sigma}}(D_{\sigma}e_{h}(t))^{2}\right) ^{\frac{1}{2}}\left(\underset{\sigma \in \mathcal{E}} {\sum} \mathrm{mes}(\sigma) d_{\sigma}(R_{\sigma}+r_{\sigma})^{2} \right)^{\frac{1}{2}},
\end{eqnarray*}
where $d_\sigma=d_{i,\sigma}$, for $\sigma \in \mathcal{E}_{i}$.
 Using the fact that $\;\underset{\sigma \in \mathcal{E}}{\sum}\mathrm{mes}(\sigma) d_{\sigma}\leqslant d \;\mathrm{mes}(\Omega) $, \lemref{consistency} and  relation \eqref{flux1} yields
\begin{eqnarray}
 \label{derr4}
\lefteqn{\underset{A_i\in \mathcal{T}}{\sum}\underset{\sigma \in \mathcal{E}_{i}}{\sum}\mathrm{mes}(\sigma)e_{i}(t)(R_{i,\sigma}(t)+r_{i,\sigma}(t))}&& \nonumber\\
&\leqslant& C_{3}\,h\,(\mathrm{mes}(\Omega)\,d)^{\frac{1}{2}}\Vert e_{h}(t)\Vert_{1,\mathcal{T}} \nonumber\\
&\leqslant& (C_{5})^{-1}C_{3}\,h\,(\mathrm{mes}(\Omega)\,d)^{\frac{1}{2}}\Vert e_{h}(t)\Vert_{1,h}.
\end{eqnarray}
For an arbitrary constant $C>0$, Young's inequality implies that
\begin{eqnarray}
\label{Yp}
 \left\lbrace \begin{array}{l}
\vert C\,h\,\underset{A_i \in \mathcal{T}}{\sum} \mathrm{mes}(A_i) e_{i}(t)\vert=\vert \underset{A_i \in \mathcal{T}}{\sum} (C\;h\; \mathrm{mes}(A_i)^{\frac{1}{2}}) (\mathrm{mes}(A_i)^{\frac{1}{2}} e_{i}(t)) \vert \\
\newline\\
\qquad \qquad \qquad \qquad \qquad\leqslant\dfrac{1}{2}\Vert e_{h}(t)\Vert_{0,h}^{2}+\dfrac{1}{2} C^{2} h^{2}\;\mathrm{mes}(\Omega)\\
\newline\\
C\,h\Vert e_{h}(t)\Vert_{1,h}\leqslant \dfrac{1}{2}C^{2} h^{2}+ \dfrac{1}{2} \Vert e_{h}(t)\Vert_{1,h}^{2}.
\end{array}\right.
\end{eqnarray}
Using expression (\ref{derr4}) and (\ref{Yp}) in expression (\ref{derr3}) yields
\begin{eqnarray}
 \label{derr5}
\left\lbrace \begin{array}{l}
\dfrac{1}{2}\left[ \underset{A_i \in \mathcal{T}}{\sum}  \mathrm{mes}(A_i)\dfrac{d(e_{i}^{2}(t))}{dt}+\Vert e_{h}(t)\Vert_{1,h}^{2}+2c_{0}\,\Vert e_{h}(t)\Vert_{0,h}^{2}\right] \leqslant (C_{7}\Vert e_{h}(t)\Vert_{0,h}^{2}+ C_{6}\,h^{2}\\
\newline\\
C_{6}=C_{6}(c_0,C_{1}, C_{3},C_{3}',C_{4},C_{5}), C_7=C_7(C_4').
\end{array}\right.
\end{eqnarray}
Bounding the left hand side of expression (\ref{derr5}) below yields 
\begin{eqnarray}
\label{derr6}
 \underset{A_i\in \mathcal{T}}{\sum}  \mathrm{mes}(A_i)\dfrac{d(e_{i}^{2}(s))}{ds}\leqslant 2C_{7}\Vert e_{h}(s)\Vert_{0,h}^{2}+ 2 C_{6}\,h^{2}, \qquad \forall s\in [0,T].
\end{eqnarray}
Integrating both sides of expression (\ref{derr6}) through interval $\left[0,t \right],\; 0\leq t\leq T $ yields
\begin{eqnarray}
\label{derr7}
 \Vert e_{h}(t)\Vert_{0,h}^{2}\leq \Vert e_{h}(0)\Vert_{0,h}^{2} +2 \,C_{6}\,T \,h^{2}+2C_{7}\int_{0}^{t}\Vert e_{h}(s)\Vert_{0,h}^{2}ds, \qquad \forall t \in [0,T].
\end{eqnarray}
Applying  the discrete Gronwall  yields
\begin{eqnarray}
\label{derr8}
\Vert e_{h}(t)\Vert_{0,h}^{2}&\leq& C \left( \Vert e_{h}(0)\Vert_{0,h}^{2}+ h^{2}\right),\\
\newline \nonumber\\
C&=&C(\mathcal{B},\Omega,X,F,\mathbf{D},\mathbf{q}, T,\zeta_{1}). \nonumber
\end{eqnarray}
%
\end{proof}

\section{Full discretization and main result}
\label{sec:etd}
\subsection{Exponential Euler method  for time discretization}
For simplicity we consider  a constant time-step $\Dt>0$.
At time $t_{m}=m \Delta t \in [ 0,T], $ the mild solution (\ref{adrd}) 
is given by
\begin{eqnarray}
\label{mmild}
X_{h}(t_{m})= S_{h}(t_{m})X_{0h}+ \int_{0}^{t_{m}} S_{h}(t_{m}-s)P_{h}F(X_{h}(s))ds.
\end{eqnarray}
 Then, given the solution $X_{h}$ at the
time $t_{m}$, we can construct the corresponding solution at $t_{m+1}$
as
\begin{eqnarray}
\label{exactnu}
X_{h}(t_{m+1})= S_{h}(\Dt)X_{h}(t_{m})
+\int_{0}^{\Dt}
 S_{h}(\Dt-s)P_{h}F(X_{h}(t_{m}+s))d s.
\end{eqnarray}
Note that the expression in \eqref{exactnu} is still an exact
form of $X_{h}$. The idea behind exponential time differencing is to
approximate $P_{h}F(X_{h}(t_{m}+s))$ by a
suitable polynomial~\cite{MC,Tr}. We consider the simplest case
where $P_{h}F(X_{h}(t_{m}+s))$ is
approximated by the constant $P_{h}F(X_{h}(t_{m}))$ and the corresponding  scheme (ETD1) is given by
\begin{equation}
\label{etd}
X _{h}^{n+1}  = e^{-\Dt A_{h}}
X_{h}^{n}+\Dt\varphi_{1} (-\Dt A_{h})P_{h}F(X_{h}^{m})
\end{equation}
 where
 $$\varphi_{1}(-\Delta t A_{h})=(-\Delta t\, A_{h})^{-1}\left( e^{-\Delta
    t A_{h}}-I\right)= \frac{1}{\Delta t}\int_{0}^{\Delta t}
e^{-(\Delta t- s)A_{h}}ds. $$ 
Note that the ETD1 scheme in (\ref{etd})  can be rewritten as
\begin{eqnarray}
\label{etd1}
X_{h}^{m+1}  =  X_{h}^{m}+\Dt\varphi_{1} (-\Dt A_{h})(-A_{h}X_{h}^{m}+P_{h}F(X_{h}^{m})).
\end{eqnarray}
This new expression has the advantage that it is  computationally more
efficient as only one matrix exponential function needs to be
evaluated at each step. 
%
\subsection{Main result}
\label{sec:th1}
To achieve the optimal orders of convergence in time and space, the solution $X$ need to be regular. 
\begin{theorem}
\label{T1}

Let  $\mathcal{B}\subset V$ be bounded, consider the solution $X$ of \eqref{adr6} and $X_{h}^{m}$ the numerical solution (\ref{etd1}) 
given by combining the finite volume method in space discretization and  ETD1 scheme in time integration. Assume that
$X(t_{k}) \in \mathcal{B}$, $X_{h}^{k} \in \mathcal{B}\bigcap  V_h $ and  $ X_h(t_k)\footnote{This is the solution of \eqref{fvfd} or \eqref{adrd}}\in \mathcal{B}\bigcap V_h $
for  all $t_{k}= k \Delta t \leq T \leq  t^{*},\, k\in \mathbb{N}$.
 We aslo assume that the unique mild solution $X$  of \eqref{adr6}  is the classical solution (i.e. $X$ is twice continuously differentiable with respect 
to $\mathbf{x}$ and differentiable with respect to $t$), \assref{meshassumption} is satisfied  and the reaction 
function $F$ satisfies  \eqref{lip}. 
Furthermore assume that $X_{0} \in \mathcal{D}(A)\bigcap \mathcal{B}$, $X_{0h} \in V_h\bigcap \mathcal{B}$  and $f(\mathbf{x},u)$ is  differentiable respect to $\mathbf{x}$ and $u$  with
\eqref{Extraf} or \eqref{extraff}, then the following estimate holds
\begin{eqnarray*}
 \Vert X(t_{m})-X_{h}^{m} \Vert_{0,h} \leq  C\left(\Vert X_0- X_{0h} \Vert_{0,h}+\Delta t+h\right),
\end{eqnarray*}
where $C=C( \mathcal{B},\Omega,X,F,\mathbf{D},\mathbf{q}, T,\zeta_{1}).$
\end{theorem}
Before give the proof, let us provide two important results.
\begin{lemma}
\label{eqnorm}
[\textbf{Norms equivalence}]\\
Consider the discrete $L^2(\Omega)$ norm  $\Vert.\Vert_{0,h}$ associated to the discrete scalar product \eqref{dproduct}or \eqref{dproductL2} and the discrete norm  $\Vert.\Vert_{0,H}$ 
defined  with the dual mesh $\mathcal{T}_h$ such that  for $v\in C(\Omega)$ 
\begin{eqnarray}
 \Vert v\Vert_{0,H}=\sqrt{\underset{K\in \mathcal{T}_h}{\sum} h_K^{d} \underset{\mathbf{x}_i\in K}{\sum}v(\mathbf{x}_i)^{2}}.
\end{eqnarray}
Assume that $\mathcal{T}$ is regular (\assref{meshassumption} is satisfied, so $\mathcal{T}_h$ is regular according to Remark \ref{regularduall}), the norms $\Vert.\Vert_{0,h}$ ,$\Vert.\Vert_{0,H}$, and the $ L^2(\Omega)$ norm
$\Vert.\Vert$ are equivalent in $V_h$ uniformly with respect to h.
\end{lemma}
\begin{proof}
See  \cite[Remark 6.16, p. 275]{EP} and \cite[Theorem 3.43, p. 163]{EP}.
\end{proof}
\begin{proposition}
\label{prop111}
$\left[\textbf{Interpolation error}\right]$

Let $ \mathcal{T}$ be an admissible mesh in the sense of \defref{admissiblemesh} and $\mathcal{T}_{h}$ its dual Delaunay triangulation 
(remember that $\left\lbrace \mathbf{x}_{i}\right\rbrace  $ are vertices of $\mathcal{T}_{h}$ and centers of $\mathcal{T}$).  
Let $I_{h}: C(\overline{\Omega})\rightarrow V_{h}$ defined by
\begin{eqnarray}
 I_{h}(u)=\underset{ i\in \mathcal{T}}{\sum} u(\mathbf{x}_{i})\varphi_{\mathbf{x}_{i}}, \qquad\qquad u \in C(\overline{\Omega})\
\end{eqnarray}
where $\left\lbrace\varphi_{\mathbf{x}_{i}}\right\rbrace_{i\in \mathcal{T}}$ is the nodal basis corresponding to $\left\lbrace \mathbf{x}_{i}\right\rbrace_{i\in \mathcal{T}}$ in the sense of finite element method $(\varphi_{x_{i}}(x_{j})=\delta_{i,j})$.
 Let $X$ the  solution of \eqref{adr6} given by \eqref{mild}.
 If $ X(t)\in H^{2}(\Omega)$, then there exists a positive constant $C_{0}>0$ independent  of $X$ and $t$ such that the following estimate holds
\begin{eqnarray}
\label{ineq1}
 \Vert X(t)- I_{h}(X(t))\Vert \leq C_{0} \vert X(t)\vert_{2}\, h^{2},
\end{eqnarray}
where $ \vert .\vert_2$ denotes  the semi norm of $ H^{2}(\Omega)$ \footnote{Note that this semi norm uses only second order derivatives which belong to $L^{2}(\Omega)$.}.
Furthermore,  if  $X\in C([0,T], H^{2}(\Omega))$,  there exists  $C_{0}=C_{0}(X,T)$ such that 
\begin{eqnarray}
\label{ineq2}
\Vert X(t)- I_{h}(X(t))\Vert \leq C_{0}(X,T) h^{2}, \qquad \qquad \forall t \in [0,T].
\end{eqnarray}
\begin{proof}
 For  \eqref{ineq1}, see \cite [Section 3.4, Theorem 3.29, page 138]{EP} or \cite[Theorem 17.1, page 132]{lions} with $k=1$ and $m=0$. Recall that
 $ C([0,T], H^{2}(\Omega))$ is the set of  continuous functions $ v: [0,T]\rightarrow H^{2}(\Omega)$ such that $\underset{t\in[0,T]}{\sup} \Vert v (t) \Vert_{2}< \infty$.
 To have  \eqref{ineq2}, we obviously have
 \begin{eqnarray}
\label{ineq22}
\Vert X(t)- I_{h}(X(t))\Vert \leq C_{0} \vert X(t)\vert_{2}\, h^{2} \leq    C_{0} \underset{t\in [0,T]}{\sup} \Vert X (t) \Vert_{2}  h^{2} = C_{0}(X,T) h^{2}, \quad \forall t \in [0,T].
\end{eqnarray}
\end{proof}
\end{proposition}
As the preparatory results are provided, let us proof our main result (\thmref{T1}). 

\begin{proof}
We use the equivalence of the norms $\Vert.\Vert $ and $\Vert.\Vert_{0,h}$ in $V_{h}$ as  we have assumed that the mesh $\mathcal{T}$ is regular (see \lemref{eqnorm}).
 Using the triangle inequality yields
\begin{eqnarray}
 \Vert X(t_{m})-X_{h}^{m}\Vert_{0,h} 
 &\leq& \Vert X(t_{m})-X_{h}(t_{m})\Vert_{0,h} + \Vert X_{h}(t_{m})-X_{h}^{m}\Vert_{0,h} \nonumber\\
  &=& I +II.
\end{eqnarray}

As  $I$ is already estimated in \thmref{Ft}, let us estimate $II$. From  \eqref{mmild} and \eqref{etd}, we have
\begin{eqnarray}
  X_{h}(t_{m})=S_{h}(t_{m}) X_{0h}+\underset{k=0}{\sum^{m-1}}\int_{t_{k}}^{t_{k+1}} S_{h}(t_{m}-s)P_{h}F(X_{h}(s))ds,
\end{eqnarray}
and 
\begin{eqnarray}
 X_{h}^{m}=S_{h}(t_{m}) X_{0h}+\underset{k=0}{\sum^{m-1}}\int_{t_{k}}^{t_{k+1}} S_{h}(t_{m}-s)P_{h}F(X_{h}^{k})ds.
\end{eqnarray}
The smoothing properties of the semigroup $S_{h}$ in \propref{prop1}
and the equivalence $\Vert.\Vert \equiv \Vert.\Vert_{0,h}$ in $V_{h}$ yields
\begin{eqnarray}
 \Vert X_{h}(t_{m})- X_{h}^{m} \Vert_{0,h}&\equiv& \Vert X_{h}(t_{m})- X_{h}^{m} \Vert\nonumber \\
       &\leq& \underset{k=0}{\sum^{m-1}}\int_{t_{k}}^{t_{k+1}} \Vert  S_{h}(t_{m}-s)P_{h} \left(F(X_{h}(s))-F(X_{h}^{k})\right)\Vert ds.  \nonumber
 \end{eqnarray}
 Following \cite{Stig} we should prove  that
\begin{eqnarray}
\label{keyd}
 \Vert A_h^{-1/2}P_{h} v\Vert \leq  Ch \Vert v\Vert +\Vert v\Vert_{-1},\qquad  \forall v \in L^2(\Omega).
\end{eqnarray}
From \cite[(18)]{florin} it is well known that 
\begin{eqnarray}
\label{scadif}
 \vert (u,v)-\langle u,v\rangle_{0,h}\vert \leq Ch^{s+p}\Vert u \Vert_s \Vert v \Vert_p,\quad \forall  u \in H^s(\Omega),\, v\in H^p(\Omega),\qquad s,p \in \{0,1\}.
\end{eqnarray}

Indeed identifying  $L^2(\Omega)$ to its dual, as $A_h^{*-1/2}$ is uniformly bounded, using \eqref{scadif} and  the definition of  the projection $P_h$ yields
{\small {
\begin{eqnarray}
 \Vert A_h^{-1/2}P_{h} v\Vert &=& \underset{u_h \in V_h}{\sup} \dfrac{\vert (A_h^{-1/2}P_{h} v,u_h)\vert}{\Vert u_h \Vert_1} \nonumber\\
    &=& \underset{u_h \in V_h}{\sup} \dfrac{\vert (P_{h}v,A_h^{*\,-1/2} u_h) \vert}{\Vert u_h \Vert_1}\nonumber\\
    &\leq & C_1 h \Vert v \Vert+  \underset{u_h \in V_h}{\sup} \dfrac{\vert \langle P_{h}v,A_h^{*\,-1/2} u_h\rangle_{0,h} \vert}{\Vert u_h \Vert_1}\nonumber\\
    &= & C_1 h \Vert v \Vert+  \underset{u_h \in V_h}{\sup} \dfrac{\vert \langle v,A_h^{*\,-1/2} u_h\rangle_{0,h} \vert}{\Vert u_h \Vert_1}\nonumber\\
    &\leq & C h \Vert v \Vert+ \underset{w_h \in V_h}{\sup} \dfrac{\vert ( v, w_h)\vert}{\Vert A_h^{*1/2} w_h \Vert_1}\nonumber\\
    &\leq & C h \Vert v \Vert + \Vert v\Vert_{-1}.
 \end{eqnarray}
 }}
Comparing  with the results in \cite{Stig}, the estimation \eqref{keyd} clearly shows the difference between the finite element method 
 and the finite element method  while performing the space discretization of problem of type \eqref{adr}. 
 
 Let us back in our main proof.  Using  \eqref{keyd}, \propref{prop1}, the  Lipschitz conditions \eqref{lip} and \eqref{lip1} and  the fact  that  both  the full discrete  and semi discrete solutions are in
 $\mathcal{B}$ allows to have
 \begin{eqnarray}
 \lefteqn{\Vert X_{h}(t_{m})- X_{h}^{m} \Vert_{0,h}} \nonumber\\
       &\leq& \underset{k=0}{\sum^{m-1}}\int_{t_{k}}^{t_{k+1}} \Vert  S_{h}(t_{m}-s)A_h^{1/2} A_h^{-1/2}P_{h} \left(F(X_{h}(s))-F(X_{h}^{k})\right)\Vert ds,  \nonumber\\
       &\leq& \underset{k=0}{\sum^{m-1}}\int_{t_{k}}^{t_{k+1}} (t_m-s)^{-1/2} \left(C h \Vert F(X_{h}(s))-F(X_{h}^{k}) \Vert+ \Vert F(X_{h}(s))-F(X_{h}^{k})\Vert_{-1}\right) ds \nonumber\\
       &\leq& \underset{k=0}{\sum^{m-1}}\int_{t_{k}}^{t_{k+1}} (t_m-s)^{-1/2} \left(C h \Vert X_{h}(s)-X_{h}^{k} \Vert_1+ \Vert X_{h}(s)-X_{h}^{k}\Vert\right) ds\nonumber\\
        &\leq& C(\mathcal{B})h+ \underset{k=0}{\sum^{m-1}}\int_{t_{k}}^{t_{k+1}} (t_m-s)^{-1/2} \left(\Vert X_{h}(s)-X_{h}^{k}\Vert \right) ds\nonumber\\
        &\leq& C(\mathcal{B}) h+  C(\mathcal{B})\underset{k=0}{\sum^{m-1}}\int_{t_{k}}^{t_{k+1}} (t_m-s)^{-1/2} \left(\Vert X_{h}(s)-X(s) \Vert \right) ds\nonumber\\
         &+& C(\mathcal{B})\underset{k=0}{\sum^{m-1}}\int_{t_{k}}^{t_{k+1}} (t_m-s)^{-1/2} \left(\Vert X(s)-X_{h}^{k}\Vert \right) ds \nonumber\\
         &=& C(\mathcal{B})\left( h +II_{1}+II_{2}\right). 
 \end{eqnarray}
For $s \in [0,T]$, as the solution $X$ is  assumed to be regular, \propref{prop111}  yields
\begin{eqnarray}
   \Vert X_{h}(s)-X(s)\Vert 
          &\leq& \Vert X_{h}(s)-I_{h}(X(s))+ I_{h}(X(s))-X(s) \Vert\nonumber\\
          &\leq& \left(\Vert X_{h}(s)-I_{h}(X(s))\Vert + \Vert I_{h}(X(s))-X(s)\Vert\right)\nonumber\\
         &\leq& \left(\Vert X_{h}(s)-I_{h}(X(s))\Vert + C_{0}(X,T) h^{2} \right).
\label{errorii}
\end{eqnarray}
Since $X_{h}(s)-I_{h}(X(s)) \in V_{h}$, the equivalence $\Vert.\Vert \equiv \Vert .\Vert_{0,h}$ and the uniform estimate of the term $I$ in $[0,T]$ yields
\begin{eqnarray}
 \Vert X_{h}(s)-I_{h}(X(s))\Vert &\equiv&
     \Vert X_{h}(s)-I_{h}(X(s))\Vert_{0,h} \nonumber\\
     &= & \Vert  X_{h}(s)- X(s) \Vert_{0,h} \quad \quad(\text{by definition of}\; \Vert. \Vert_{0,h},\, \text{and}\,\, I_h)\nonumber\\ 
       &\leq & C(\mathcal{B},\Omega,X,F,\mathbf{D},\mathbf{q},\zeta_{1}) (\Vert X_{0h}-X_0\Vert_{0,h}+h),
\label{errori}
\end{eqnarray}
which yields
\begin{eqnarray}
 II_{1} &\leq&  C(\mathcal{B}) \left(\Vert X_{0h}-X_0\Vert_{0,h}+h\right)\underset{k=0}{\sum^{m-1}}\int_{t_{k}}^{t_{k+1}}  (t_{m}-s)^{-1/2}ds + C_{0}(X,T) h^{2}\\
   &\leq&  C(\mathcal{B}) (\Vert X_{0h}-X_0\Vert_{0,h}+h) \left( \underset{k=0}{\sum^{m-1}}\int_{0}^{T}  (t_{m}-s)^{-1/2} ds\right)  + C_{0}(X,T) h^{2} \\
   &\leq & C(\mathcal{B},\Omega,X,F,\mathbf{D},\mathbf{q}, T,\zeta_{1})\left(\Vert X_{0h}-X_0\Vert_{0,h}+h\right).   
\end{eqnarray}
We also have
\begin{eqnarray}
 II_{2} &=&\underset{k=0}{\sum^{m-1}}\int_{t_{k}}^{t_{k+1}}(t_m-s)^{-1/2} \Vert X(s)- X_{h}^{k}\Vert ds\nonumber \\
  &\leq&  \underset{k=0}{\sum^{m-1}}\int_{t_{k}}^{t_{k+1}}(t_m-s)^{-1/2} \Vert X(s)- X(t_{k})\Vert ds \nonumber \\
   &+&  \underset{k=0}{\sum^{m-1}}\int_{t_{k}}^{t_{k+1}} (t_m-s)^{-1/2} \Vert X(t_{k})- X_{h}^{k}\Vert ds\nonumber \\
       &=&II_{2}^{1}+II_{2}^{2}.
\end{eqnarray}
Using Lemma \ref{lemmad1}  yields 
\begin{eqnarray}
 II_{2}^{1} &\leq& C(\mathcal{B})\underset{k=0}{\sum^{m-1}}\int_{t_{k}}^{t_{k+1}}(t_m-s)^{-1/2}(s-t_{k}) ds\nonumber\\
             &\leq& C(\mathcal{B}) \Dt\underset{k=0}{\sum^{m-1}}\int_{t_{k}}^{t_{k+1}}(t_{m}-s)^{-1/2}ds\nonumber\\
             &\leq& C(\mathcal{B})\Dt \int_{0}^{T}(t_{m}-s)^{-1/2}ds \nonumber\\
             &\leq& C(\mathcal{B},T) \Dt.
\end{eqnarray}
 Using the equivalence $\Vert.\Vert\equiv \Vert.\Vert_{0,h}$
in $V_{h}$ (\lemref{eqnorm}) allow to have
\begin{eqnarray}
 II_{2}^{2}&\leq & \underset{k=0}{\sum^{m-1}}\int_{t_{k}}^{t_{k+1}}(t_m-s)^{-1/2}\Vert X(t_{k})- I_{h}(X(t_{k}))+I_{h}(X(t_{k}))- X_{h}^{k}\Vert ds\\
&\leq &  \underset{k=0}{\sum^{m-1}}\int_{t_{k}}^{t_{k+1}} (t_m-s)^{-1/2}\Vert X(t_{k})- I_{h}(X(t_{k}))\Vert +\Vert I_{h}(X(t_{k}))- X_{h}^{k}\Vert_{0,h} ds\\
             &\leq &C(X,T) \left( h^{2} + \underset{k=0}{\sum^{m-1}}\int_{t_{k}}^{t_{k+1}}(t_m-s)^{-1/2}\Vert X(t_{k})-X_{h}^{k}\Vert_{0,h}ds\right) 
\end{eqnarray}
Then
\begin{eqnarray*}
 II &\leq& C(\mathcal{B}) \left(\left( \Vert X_{0h}-X_0\Vert_{0,h}+\Dt +h\right)+ \underset{k=0}{\sum^{m-1}}\int_{t_{k}}^{t_{k+1}}(t_m-s)^{-1/2}\Vert  X(t_{k})-X_{h}^{k}\Vert_{0,h} ds\right),\\
 \end{eqnarray*}
Combining estimates $I$ and  $II$ yields
{\small {
\begin{eqnarray}
\label{lastest}
  \lefteqn{\Vert X(t_{m})-X_{h}^{m})\Vert_{0,h} } && \nonumber\\
&\leq& C(\mathcal{B}) \left( \Vert X_{0h}-X_0\Vert_{0,h}+\Dt +h+ \underset{k=0}{\sum^{m-1}}\int_{t_{k}}^{t_{k+1}}(t_m-s)^{-1/2}\Vert X(t_{k})-X_{h}^{k}\Vert_{0,h} ds\right),
\end{eqnarray}
}}
where  $C(\mathcal{B})=C(\mathcal{B},\Omega,X,F,\mathbf{D},\mathbf{q},\zeta_{1})$.
Applying the generalized discrete Gronwall in (\ref{lastest}) ends the proof.
\end{proof}
\begin{remark}
 Using \propref{prop111}, if the initial solution $X_0 \in \mathcal{D}(A) \subset  H^{2}(\Omega) \subset C(\Omega)$,  and $X_{0h}=I_{h}(X_0)$, we obviously have  the following estimate
 \begin{eqnarray}
 \label{opti}
  \Vert X(t_{m})-X_{h}^{m})\Vert_{0,h} \leq C(\mathcal{B}) \left (\Dt +h\right).
 \end{eqnarray} 
 \end{remark}
 \begin{remark}
 Although optimal orders in time and space  are achieved for  $X_0 \in \mathcal{D}(A)$, this condition is not enough to ensure  the   
 regularity of  the solution as stated in  \thmref{T1} and \thmref{Ft}, which requires that $X_0 \in C^2(\Omega)$.
\end{remark}

\section{Numerical simulations}
\label{sim}
The goal here is to illustre the theoretical result \eqref{opti} and to compare  ETD1 scheme with  standard time-stepping methods, implicit Euler and semi implicit schemes. 
The difference between these schemes is the time integration as  the space discretization is performed using  finite  volume method.  
So our main focus will be the errrors in time. In \cite[Section 4, Figure 2]{Antoine},
the convergence in space has been studied for linear with exact solution and the optimal order in space \eqref{opti} has been reached. 
Here we compare the  time errors and the efficiency of the schemes.

Our code was implemented in Matlab 7.7.
In the legends of all of our graphs we use the following notation
\begin{itemize}
  \item ``Implicit with Newton'' denotes results from the implicit Euler with standard Newton's method.
  \item ``Implicit with Newton V'' denotes results from the implicit Euler with a variant of Newton's method where  the Jacobian is kept constant \cite{Antoine}.
  \item ``\Leja ETD1'' denotes results from ETD1 with real fast \Leja points for matrix exponential.
  \item ``Krylov ETD1'' denotes results from  ETD1 with Krylov subspace for matrix exponential.
  \item ``Semi implicit'' denotes results from the semi-implicit scheme.
\end{itemize}
We now evaluate the ETD1 method for a non-linear ADR problem where the
non-linear reaction term is given by $f( \mathbf{x},u)=-\theta u^{2}(1-u)$. 
We
take $\theta =100$, use a constant velocity of
$\underbar{q}=[-0.01,-0.01]^{T}$, and the dispersion tensor has the
entries $D_{1}=D_{2}=10^{-4}$. The domain is
$\Omega=\left[ 0,1\right)\times \left[ 0,1\right)$, which we
discretise with $h=\Delta x= \Delta y=10^{-2}$.
We can observe that $f$ satisfies the local Lipschitz condition \eqref{lipff}, so  existence and uniqueness of the local solution is ensured.
Indeed the global solution exists and is given by \cite{exact1}
\begin{equation}
\label{excats}
  C(x,y,t)=\left( 1+\exp\left( a(x+y-bt)+a(b-1)\right)\right)^{-1}
\end{equation}
where $a=\sqrt{\theta/\left(4\times10^{-4}\right)}$ and $b=-0.02+\sqrt{\theta\times 10^{-4}}$.
The initial condition and boundary conditions are defined with respect to the exact solution \eqref{excats}.

\begin{figure}[!ht]
  \subfigure[]{
    \label{FIG03a}
    \includegraphics[width=0.48\textwidth]{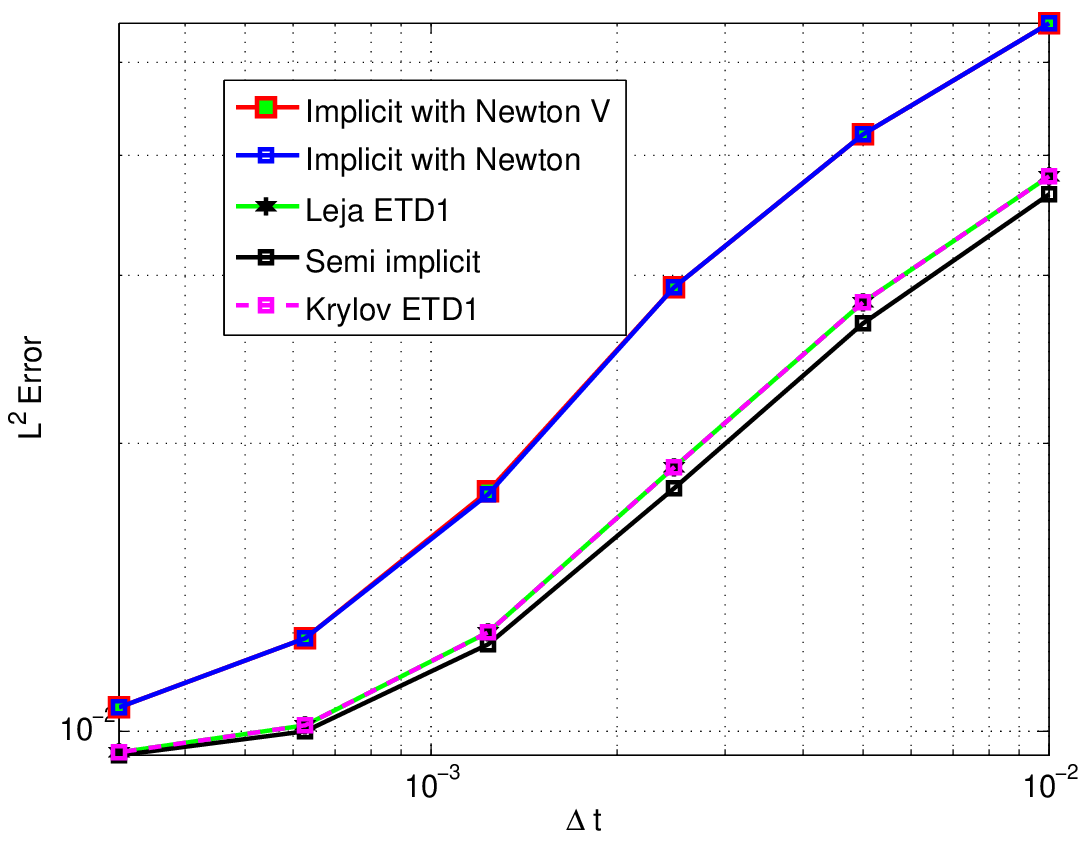}}
  \hskip 0.01\textwidth
  \subfigure[]{
    \label{FIG03b}
    \includegraphics[width=0.48\textwidth]{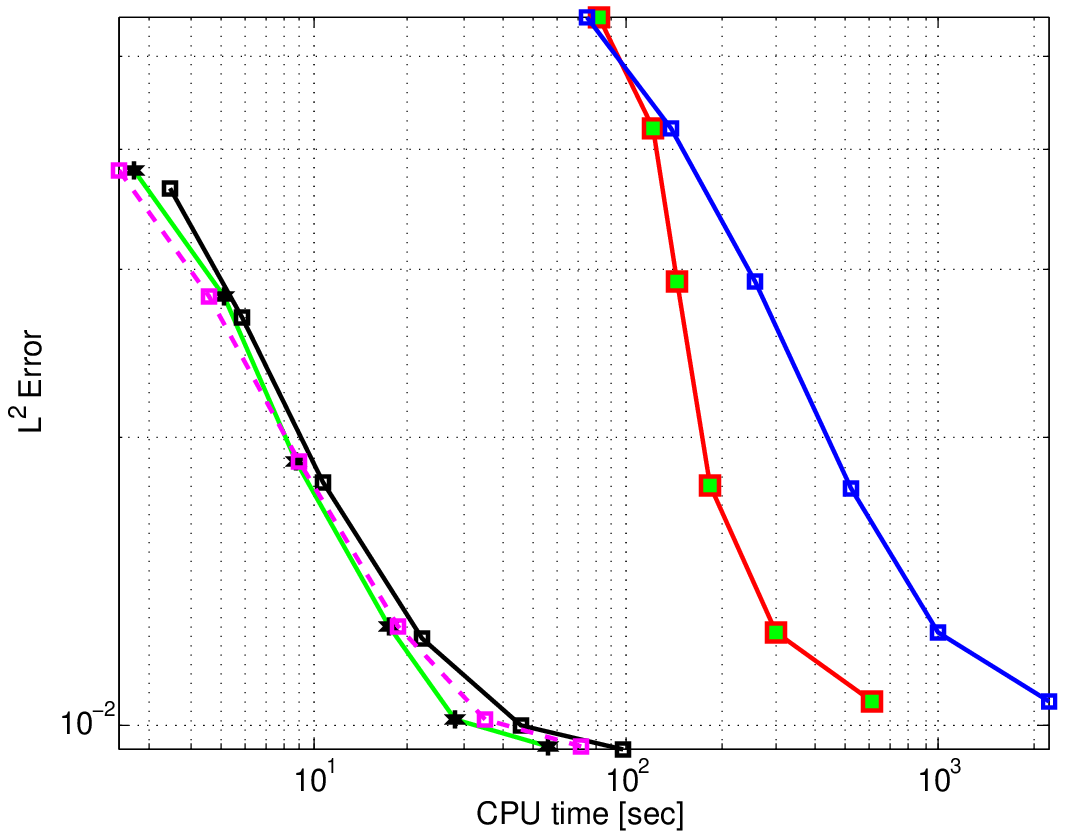}}
  \caption{(a) Convergence of the $L^2$ norm at $T=1$ as a function of $\Dt$. (b) The $L^2$ norm at $T=1$ as a function of CPU time. Both are for the the non-linear ADR in homogeneous porous media (Problem 2).}
  \label{FIG03}
\end{figure}

\figref{FIG03a} shows
the convergence as a function of the chosen time-step $\Dt$, measuring
the error at the final time $T=1$. The semi-implicit time-stepping
method and the ETD1 methods have similar error constants. All schemes
have the same rate of convergence  $\mathcal{O}(\Dt)$, which is predicted in our convergence result \eqref{opti}.
\figref{FIG03b} shows the $L^{2}$ error as a function of CPU time,
which is given in \figref{FIG03a}. ETD1 graphs are also similar
to the semi-implicit one. However, all three methods, ETD1
with Leja points and Krylov subspace technique and semi-implicit
time-stepping, outperform the implicit time-stepping methods.
\section*{Acknowledgements}
 This work was supported by the Overseas Research Students Awards Scheme (ORSAS) at Heriot
Watt University and  Robert Bosch Stiftung through the AIMS ARETE chair programme.

\end{document}